\documentclass[11pt]{amsart}
\usepackage{amssymb,amsthm,xcolor}
\usepackage{eucal}
\usepackage[shortlabels]{enumitem}
\usepackage{thmtools,thm-restate}
\usepackage{hyperref}
\numberwithin{equation}{section}
\setlist[description]{font=\normalfont}

\usepackage{mathtools}

\declaretheorem[numberwithin=section]{theorem}
\newtheorem*{theorem*}{Theorem 1.5}  %%%% dirty trick
\declaretheorem[sibling=theorem]{proposition}

\declaretheorem[sibling=theorem]{lemma}

\declaretheorem[sibling=theorem]{corollary}
\declaretheorem[style=definition,sibling=theorem]{definition}

\declaretheorem[style=remark,sibling=theorem]{remark}
\declaretheorem[sibling=theorem]{fact}

\DeclareMathOperator{\cf}{cf}
\DeclareMathOperator{\cof}{cof} 
\DeclareMathOperator{\ot}{ot}

\DeclareMathOperator{\dom}{dom}

\DeclareMathOperator{\cl}{cl}

\DeclareMathOperator{\range}{range}

\DeclareMathOperator{\stem}{stem}
\DeclareMathOperator{\tr}{tr}
\DeclareMathOperator{\spl}{spl}

\DeclareMathOperator{\osucc}{osucc}
\DeclareMathOperator{\Split}{split}

\DeclareMathOperator{\acc}{acc}\DeclareMathOperator{\nacc}{nacc}
\newcommand{\seq}[2]{\langle #1 : #2 \rangle}

\newcommand{\Dl}{\textup{\textsf{Dl}}}
\newcommand{\w}{\wedge}

\newcommand{\rest}{\upharpoonright}

\newcommand{\la}{\langle}
\newcommand{\ra}{\rangle}
\newcommand{\such}{:}  

\newcommand{\concat}{{}^\smallfrown{}}
\newcommand{\nothing}[1]{} 
\newcommand{\trl}{\trianglelefteq}

\newcommand{\cF}{\mathcal F}

\newcommand{\cD}{\mathcal D}

\newcommand{\B}{\mathbb B}

\newcommand{\bP}{\mathbb P} %% for 1191
\newcommand{\bQ}{\mathbb Q}
\newcommand{\bV}{\mathbf V}
\newcommand{\bG}{\mathbf G}

\newcommand{\eps}{\varepsilon}
\DeclareMathOperator{\club}{club}

\newcount\skewfactor
\def\mathunderaccent#1#2 {\let\theaccent#1\skewfactor#2
\mathpalette\putaccentunder}
\def\putaccentunder#1#2{\oalign{$#1#2$\crcr\hidewidth
\vbox to.2ex{\hbox{$#1\skew\skewfactor\theaccent{}$}\vss}\hidewidth}}
\def\name{\mathunderaccent\tilde-3 }

\newcommand{\qa}{{\mathbb Q}^{\rm Miller}_\kappa} %classical Miller
\newcommand{\qb}{{\mathbb Q}^{\rm Sacks}_\kappa} %classical Sacks
\newcommand{\qbW}{{\mathbb Q}^{\rm Sacks}_{(\kappa,W)}}
\newcommand{\qc}{{\mathbb Q}^{\rm Silver}_\kappa} %classical Sacks
\newcommand{\qd}{{\mathbb Q}^{\rm Laver}_\kappa} %classical Sacks

\newcommand{\xname}{x}
\newcommand{\checki}{\check{I}}
\begin{document}
\date{April 7, 2025}%, compiled \today}

\title[Forcing Diamond]{Forcing Diamond and Applications to Iterability}
\author{Heike Mildenberger}
\address{Mathematisches Institut, Albert-Ludwigs-Universit\"at Freiburg,  Ernst-Zermelo-Str. 1, 79104 Freiburg, Germany}
\email{heike.mildenberger@math.uni-freiburg.de}
\author{Saharon Shelah}
\address{Einstein Institute of Mathematics, The Hebrew University of Jerusalem, Edmond Safra Campus Giv'at Ram, Jerusalem, Israel}
\email{shelah@math.huji.ac.il}

\begin{abstract} 
We show that higher Sacks forcing at a regular limit cardinal and club Miller forcing at an uncountable regular cardinal  both add a diamond sequence. 
 We answer the longstanding question, whether $\kappa = \kappa^{<\kappa} \geq\aleph_1$ implies that $\kappa$-supported iterations of $\kappa$-Sacks forcing do not collapse $\kappa^+$ and are $\kappa$-proper in the affirmative. The results pertain to other higher tree forcings. 
\end{abstract} 
\thanks{The research of the second author is partially support by the Israel Science Foundation (ISF) grant no: 2320/23. This is number 1259 in his list. The authors are grateful to Craig Falls for generously funding typing services that were used during the work on the paper.}
\makeatletter
\@namedef{subjclassname@2020}{\textup{2020} Mathematics Subject Classification}
\makeatother

\maketitle

\section{Introduction}

Tree forcings like Silver forcing, Sacks forcing, Miller forcing or Laver forcing are used to arrange combinatorial properties of the power set of $\mathbb R$. Baumgartner \cite{baumg:a-d}, Kanamori \cite{Kanamori-higher-sacks} and later many researchers found analogues for an uncountable regular cardinal $\kappa$ instead of $\omega$ that share at least part of the properties of their relatives at $\omega$. The extent of the analogy depends on properties of $\kappa$. Here we are mainly interested in conditions that ensure the preservation of $\kappa^+$ and a version of $\kappa$-properness (see \autoref{kappaproper}) for iterations with supports of size $\leq \kappa$.

Baumgartner \cite[Theorem 6.7]{baumg:a-d} showed that the $\kappa$-supported product of $\kappa$-Silver forcing does not collapse $\kappa^+$ under $\diamondsuit_\kappa$.
Kanamori showed that iterating $\kappa$-Sacks forcing with supports of size $\leq \kappa$ does not collapse $\kappa^+$ if $\diamondsuit_\kappa$  \cite[Theorem 3.2]{Kanamori-higher-sacks} holds or if $\kappa$ is strongly inaccessible \cite[Section 6]{Kanamori-higher-sacks}. The same proofs work also for numerous $(<\kappa)$-closed forcings in which forcing conditions are trees with club many splitting nodes. Iterations may be replaced by $(\leq \kappa)$-supported products \cite[Section 5]{Kanamori-higher-sacks}.

Shelah \cite{Sh:922} showed that $\kappa^{\kappa} = \kappa = \lambda^+\geq \aleph_2$ implies
$\diamondsuit_\kappa(\kappa \cap \cof(\mu))$ for any regular $\mu \neq \cf(\lambda)$.
%The cofinality $\cf(\lambda)$ is called the critical cofinality.
Hence for successor cardinals $\kappa = \kappa^{<\kappa} \geq \aleph_2$, the conditions that Baumgartner and Kanamori used for their iterability proofs are fulfilled.

In \cite{MdSh:1191} we showed that in Kanamori's iterability theorem (see \autoref{kanamori_theorem} below) the condition  ($\diamondsuit_\kappa$ or $\kappa$ is inaccessible) can be replaced by the slightly weaker $(\Dl)_\kappa$ (see \autoref{Dl}).  There are regular limit cardinals $\kappa=\kappa^{<\kappa}$ with $\neg (\Dl)_\kappa$, see  \cite{Golshani1607.00751}.

Here we show that $\kappa = \kappa^{<\kappa} \geq \omega_1$ suffices 
as a premise for $\kappa$-properness and not collapsing $\kappa^+$ in 
$\leq \kappa$-supported iterations of higher Sacks forcing. 
We do this by showing that $\diamondsuit_\kappa$ is forced by the first  iterand of the respective forcings. A particularly simple case is \autoref{main_theorem_easy_case} for $\kappa$ weakly Mahlo.
A more complex relative of \autoref{main_theorem_easy_case} is \autoref{b14} for any regular limit $\kappa$. The latter combined with  
\autoref{l5} for $\kappa = \aleph_1$ and for $\kappa$ that are not strongly inaccessible proves \autoref{cor2}. In the case of \autoref{l5} we can work with a fixed stationary set of potential splitting levels, and we will explain this by introducing $W$-versions of the tree forcings in \autoref{l2}.

Our second  result is:
For club Miller and for club Laver forcing, the premise $\aleph_1 \leq \kappa^{<\kappa} = \kappa$ suffices for forcing a diamond and ensures iterability, see \autoref{main_theorem_3}.

For regular uncountable $\kappa$ under $\kappa^{<\kappa} > \kappa$ the forcing $\qb$ collapses $\kappa^+$ by \cite[Section 4]{MdSh:1191}. The combinatorial background \autoref{Bernstein_lemma} of \autoref{l5}  yields  in \autoref{l4(2)} another type of names for collapsing functions under $\kappa^{<\kappa} > \kappa$ for regular $\kappa$ that works also for the $W$-variants $\qbW$ (see \autoref{l2}). We do not consider singular $\kappa$ here. For singular $\kappa$, higher tree forcings share features of Namba forcing, see, e.g. the Namba trees of height $\omega_1$ used in \cite{Levine2024}.

Our first two theorems pertain to limit cardinals $\kappa$.

 \begin{theorem}\label{main_theorem_easy_case} 
  If $\kappa$ is weakly Mahlo and $S \subseteq \{\delta < \kappa: \delta \mbox{ regular limit}\}$ is stationary in $\kappa$, then $\qb \Vdash \diamondsuit_\kappa(S)$.
  The same holds for $\qc$. \end{theorem}

 \begin{theorem}\label{b14} 
  If $\kappa$ is a regular limit cardinal, $\mu < \kappa$ is a regular cardinal, and $S \subseteq \kappa \cap \cof(\mu)$ is stationary in $\kappa$, then $\qb \Vdash \diamondsuit_\kappa(S)$.
  The same holds for $\qc$. \end{theorem}

We note that for any regular limit $\kappa$, $\{\alpha < \kappa : |\alpha| = \alpha > \cf(\alpha)\}$ is stationary in $\kappa$, and hence for some regular $\mu< \kappa$,   $\{\alpha < \kappa : \alpha = |\alpha|\} \cap \cof(\mu)$ is stationary in $\kappa$. This shows that \autoref{main_theorem_easy_case} is not relevant for the proof of  \autoref{cor2}, since for any regular limit $\kappa$ with $\kappa^{<\kappa} = \kappa$ we can invoke \autoref{b14}.

Working with the approachability ideal, we present a different name of a diamond in Section \ref{S6}, and 
with this we settle the case of $\kappa = \aleph_1$ in Kanamori's question.
We apply following theorem with $\kappa = \aleph_1$ under {\sf CH} with $\sigma = \aleph_0$. The theorem also shows that in the case of $\kappa$ being a successor cardinal $\kappa = \lambda^+$, a diamond on a stationary subset of $\kappa \cap \cof(\cf(\lambda))$ is forced.

The subforcings $\qbW$ from \autoref{l2}, $W\subseteq \kappa$, $W$ stationary,  of $\qb$ respect weaker demands on splitting nodes than $\qb$ and are still $(<\kappa)$-complete. The forcing $\qbW$ for $W = \kappa$ is $\qb$.

The names of the diamonds in the next theorem are based on approachability and Bernstein combinatorics. 

\begin{theorem}\label{l5}  Assume that $\kappa ^{<\kappa} = \kappa \geq \aleph_1$.
 Let $W \subseteq \kappa$ be stationary. 
 Suppose there are $\sigma < \kappa$ and $S \subseteq \kappa$ with the following properties:
  \begin{enumerate}
 \item[$(\ast)$]
 \begin{enumerate}
 \item[(a)] $\kappa = \kappa^{<\kappa} = 2^\sigma$, 
 \item[(b)] $2^{<\sigma} < \kappa$, and
 \item[(c)] $S \subseteq \kappa \cap \cof(\cf(\sigma))$, and $S$ is stationary in $\kappa$ and $S \in \checki[\kappa]$.
 \end{enumerate}
 \end{enumerate}
Then  $\qbW \Vdash \diamondsuit_\kappa(S)$.
 \end{theorem}

Analogues of \autoref{l5} for the $W$-variants of $\kappa$-Silver forcing, club $\kappa$-Miller forcing and Laver forcing  hold. The approachability ideal $\checki[\kappa]$ is reviewed in Subsection \ref{S2.3}.

We recall:

 \begin{theorem}[Kanamori, \cite{Kanamori-higher-sacks}]\label{kanamori_theorem}
        Assume $\kappa^{<\kappa}=\kappa \geq \aleph_1$.
        Let $\gamma$ be an ordinal and let $\bP = \la \bP_\alpha, \name{\bQ}_\beta \such \alpha \leq \gamma, \beta< \gamma\ra$ be a ($\leq \kappa$)-support iteration such that for $\beta < \gamma$,  $\bP_{\beta} \Vdash \name{\bQ}_\beta = \qb$. Assume $\diamondsuit_\kappa$.   Then $\bP_{\gamma}$ has the following properties.
        \begin{enumerate}
        \item[(1)]
         $\bP_{\gamma}$ does not collapse $\kappa^+$.

          \item[(2)] $\bP_\gamma$ is $\kappa$-proper.
  \end{enumerate}
   \end{theorem}

Combining Kanamori's theorem with 
 \autoref{b14}, \autoref{l5} and with Shelah's result on the diamond on successor cardinals $\lambda^+$ with $2^\lambda = \lambda^+$ \cite{Sh:922}, we derive the following.

\begin{corollary} \label{cor2} \autoref{kanamori_theorem} holds without the assumption of $\diamondsuit_\kappa$ in the ground model.
\end{corollary}
 
This answers Kanamori's question from \cite{Kanamori-higher-sacks}.
It applies to Silver, Miller and Laver forcing at $\kappa$ as well. It applies to the $W$-variants if ($\ast$) of \autoref{l5} holds.

\medskip

 Our next theorem shows that
  for club $\kappa$-Miller/Laver forcing,  for any uncountable $\kappa$ with $\kappa^{<\kappa} = \kappa$ there is a  name of a diamond that is much simpler than the names used in \autoref{b14} and \autoref{l5}.  Part (3)  answers  \cite[Question 2.17]{MdSh:1191}. In the case of $\kappa$ being strongly inaccessible iterability was proved by Kanamori \cite[Section 6]{Kanamori-higher-sacks} for the Sacks version, and by Friedman and Zdomskyy work \cite{friedman-zdomskyy} for the Miller version.  
Part (3) of the following theorem extends their result  to uncountable $\kappa$ with $\kappa^{<\kappa} = \kappa$. 
   
  \begin{theorem}\label{main_theorem_3}
        Assume $\kappa^{<\kappa}=\kappa \geq \aleph_1$. 
  \begin{enumerate}
  \item[(1)] Both $\qa$ and $\qd$ force $\diamondsuit_\kappa$.

\item [(2)] If $S \subseteq \kappa$ is stationary, then $\qa$ forces $ \diamondsuit_\kappa(S)$ and the same holds for $\qd$.
\item[(3)] The iterability theorem holds as in \autoref{cor2}.
\end{enumerate}
\end{theorem}

The article \cite{Khomskii_higher_laver} by Khomskii et.\ el.\ focuses on higher Laver forcing.

 \subsection*{Organisation of the paper}
In Section \ref{S2} we review definitions.
 In Section \ref{S3} we proof \autoref{main_theorem_easy_case} and we show that diamond in the one-step-extension leads to \autoref{cor2}.
In Section \ref{S4} we proof \autoref{b14}.
 In Section \ref{S5} we prove \autoref{main_theorem_3}.
In Section \ref{S6} we introduce $\qbW$ and prove \autoref{l5} and related results about collapsing functions  for $\qbW$. 
\section{Background}\label{S2}

  Now we review the mentioned notions.

\subsection{Combinatorics and Properness}\label{S2.1}

 \begin{definition}
   	Let $\kappa$ be a cardinal.
   	For a regular cardinal $\mu < \kappa$, we let $\kappa \cap \cof(\mu)$  stand for
   	$\{\alpha \in \kappa : \cf(\alpha) = \mu\}$. 
   \end{definition}
   
\begin{definition}
Let $\kappa$ be a cardinal of uncountable cofinality and let $S$ be a stationary subset of $\kappa$. We let $\diamondsuit_\kappa(S)$ be the following statement: There is a sequence $\seq{d_\delta}{\delta \in S}$ such that $d_\delta \in {}^\delta 2$ and such that for any $x\in {}^\kappa 2$ the set $\{\delta \in S: d_\delta = x \rest \delta\}$ is stationary. For $\diamondsuit_\kappa(\kappa)$ we write just $\diamondsuit_\kappa$.
\end{definition} 

We recall a weakening of the diamond, called $\Dl$.

\begin{definition}[See \cite{Sh:107,Sh:589,Sh:829}]\label{Dl} For a regular uncountable $\kappa$ we let $(\Dl)_\kappa(S)$ mean the following: There is a sequence $\bar{\cD} = \la \cD_\delta \such \delta \in S\ra$ such that
    $\cD_\delta \subseteq {}^\delta\delta$ is of cardinality $< \kappa$ and for every $x \in {}^\kappa \kappa$ there are stationarily many
    $\delta \in S$ such that $x \rest \delta \in \cD_\delta$. For $(\Dl)_\kappa(\kappa)$ we write $(\Dl)_\kappa$.
\end{definition}  

Inaccessibility implies $(\Dl)_{\kappa}$.

\begin{definition}\label{weakly_Mahlo} An uncountable limit cardinal $\kappa$ is called \emph{weakly Mahlo} if $\kappa$ is a regular limit cardinal (i.e., $\kappa$ is weakly inaccessible) and the set of regular limit cardinals below $\kappa$ is stationary in $\kappa$.
\end{definition}

\begin{definition}\label{kappaproper} Let $\mathcal H(\theta) = (H(\theta), \in , <_\theta)$, and $N \prec {\mathcal H}(\theta)$ and $\bQ\in N$, $p \in \bQ \cap N$. A condition 
$ q$ is called $(N, \bQ)$-generic above $p$ if $q \geq p$ and for any dense subset $D$ of $\bQ$, if $D \in N$, then 
$q \Vdash \name{\bG} \cap D \cap N \neq \emptyset$.

Let $\kappa^{<\kappa} = \kappa$. A notion of forcing $\bQ$ is called $\kappa$-proper if for any sufficiently large $\theta$ there is a club (in $[H(\theta)]^{\kappa}$) of $N \prec H(\chi)$, $|N|= \kappa$ with ${}^{<\kappa}N \subseteq N $ such that: If $\kappa, p, \bQ  \in N$, and
        $p \in \bQ\cap N$, then  there is a stronger $(N,\bQ)$ generic condition $q$.
\end{definition}

\subsection{Notation for Tree Forcing}\label{S2.2}

 \begin{definition}\label{1.1} Let $\kappa$ be an infinite cardinal.
  \begin{enumerate}
  \item[(1)] We write  ${}^{\kappa>}\kappa=\{t \colon \alpha \to \kappa : \alpha < \kappa\}$. If $s,t \in {}^{\kappa>}\kappa$ we call $s$ an inital segment of $t$ if $t \rest \dom(s) = s$. 
    A \emph{tree (on $\kappa$)} is a non-empty subset of ${}^{\kappa>}\kappa$ that is closed under initial segments. We use the symbol $\trianglelefteq$ for the initial segment
relation and the symbol $\triangleleft$ for the corresponding strict relation.

\item[(2)]
  Let $T \subseteq {}^{\kappa>} \kappa$ be a tree and $s \in T$.
  We let
  \[T^{\la s \ra} = \{ t \in T \such t \trianglelefteq s \vee s \trianglelefteq t\}.
  \]
\item[(3)]
   The elements of a tree are called nodes.
  A node that has at least two immediate $\triangleleft$-successors
  in $p$ is called a splitting node of $p$. The set of splitting
  nodes of $p$ is denoted by $\Split(p)$.
\item[(4)]
  Let $T \subseteq {}^{\kappa>} \kappa$ be a tree that contains
  a splitting node. We let the trunk of $T$, $\tr(T)$, be the $\trianglelefteq$-least splitting node of $T$. 
  \end{enumerate}
\end{definition}
  
\begin{definition}[Kanamori's Higher Sacks Forcing, \cite{Kanamori-higher-sacks}]\label{qb} Let $\kappa$ be a regular cardinal
  such that $\kappa^{<\kappa} = \kappa$.
  Conditions in the forcing order $\qb$ are trees $p \subseteq {}^{\kappa>}2$ with the following additional properties:
  \begin{enumerate}
  \item[(1)] (Perfectness) For any $s \in p$ there is an extension $t \trianglerighteq s$ in $p$
such that $t$ has two immediate successors.
\item[(2)] (Closure of splitting)
For each increasing sequence of
length $< \kappa$ of splitting
nodes, the union of the nodes on the sequence is a splitting node of $p$ as well. 
\end{enumerate}
A condition $q$ is stronger than $p$ if $q \subseteq p$. 
\end{definition}
The forcing $\qb$ has a dense subset with
 the following closure property:
  For every  increasing sequence $\la t_i \such i < \lambda\ra$ of length $\lambda< \kappa$ of nodes $t_i \in p \in \qb$ we have that the
  limit of the sequence $\bigcup \{t_i \such i < \lambda\}$ is also a node in $p$.

Club Silver forcing is called $R(1,\kappa)$ in \cite[Theorem 6.7]{baumg:a-d}.

\begin{definition}[Club Silver forcing] \label{qc} Let $\kappa$ be a regular cardinal
  such that $\kappa^{<\kappa} = \kappa$.
  Conditions in the forcing order $\qc$ are partial functions $f\colon \dom(f) \to 2$ where $\dom(f)$ is a non-stationary subset of $\kappa$.

Stronger conditions are extensions of the function $f$.
\end{definition}  
Equivalently one can see a Silver condition $f$ as a set of nodes of a higher Silver tree $T_f = \{ t \in {}^{\kappa>} \kappa : t \rest \dom(f) = f \rest \dom(t)\}$.
We can restrict $\qc$ to the dense set of conditions $f$ for which \mbox{$\kappa \setminus \dom(f)$} is a club. For these $T_f$, the limit of any increasing sequence of splitting nodes is a splitting node. This shows that \autoref{main_theorem_easy_case} and \autoref{b14} work also for Silver forcing.

\begin{definition}[Club Miller Forcing/Club Laver Forcing] \label{qa} Let $\kappa$ be a regular cardinal
  such that $\kappa^{<\kappa} = \kappa$.
  Conditions in the forcing order $\qa$ are trees $p \subseteq {}^{\kappa>}\kappa$ with the following additional properties:
  \begin{enumerate}
  \item[(1)] (Club filter superperfectness) For any $s \in p$ there is an extension $t \trianglerighteq s$ in $p$
such that
\[
\osucc_p(t) := \{\alpha \in \kappa
\such t \concat \la \alpha \ra \in p \} \mbox{ contains a club in } \kappa.
\]
We require that each node
has either only one direct successor or splits into a club.

\item[(2)] (Closure of splitting) For each increasing sequence of
length $< \kappa$ of splitting
nodes, the union of the nodes on the sequence is a splitting node of $p$ as well.
\end{enumerate}
A condition $q$ is stronger than $p$ if $q \subseteq p$.

In the case of $\qd$ we strengthen (1) to the following: there is a node $s \in p$, called the trunk of $p$ and denoted as $\tr(p)$, and for any $t \in p$ with $s \subseteq t$ the set $\osucc_p(t)$ contains a club in $\kappa$.    
\end{definition}

We remark that club Miller forcing $\qa$ is well-known: Clauses (1) and (2) imply that any $p \in \qa$ has for any $s \in p$ and any height $\alpha$ some $t \in p$ with domain $\alpha$. Again, the forcing $\qa$ has a dense subset with
 the following closure property:
  For every  increasing sequence $\la t_i \such i < \lambda\ra$ of length $\lambda< \kappa$ of nodes $t_i \in p \in \qa$ we have that the
  limit of the sequence $\bigcup \{t_i \such i < \lambda\}$ is also a node in $p$.
These clauses are sometimes added to the definition, see e.g., Brendle, Brooke-Taylor, Friedman, Montoya \cite[Def.~74]{BrendleBrooke-TaylorFriedmanMontoya}, where the forcing is called  ${\mathbb{MI}}_{\kappa}^{\rm{Clubfilter}}$. 
 Friedman and Zdomskyy \cite{friedman-zdomskyy} add the requirement that the successor set of a limit splitting node is a subset of the intersection of the $\triangleleft$-preceding splitting nodes. The set of these conditions is dense in   ${\mathbb{MI}}_{\kappa}^{\rm{Clubfilter}}$. The recent article
\cite{Khomskii_higher_laver} is concerned with higher Laver forcing. 
%\subsection{On the History}

\begin{definition}
Let $\kappa, \mu$ be cardinals.
Let $p \subseteq {}^{\kappa>} \mu$ be a tree, i.e., closed downwards.
We let $[p] = \{b \in {}^\kappa \mu : \forall \alpha \in \kappa, b \rest \alpha \in p\}$. The set$[p]$ is called the rump, body or set of $\kappa$-branches of $p$. Note that $p \mapsto [p]$ is not an absolute function.
\end{definition}

 Note that $p \mapsto [p]$ is not an absolute function. Since forcing conditions are perfect trees, in the generic extension there are new branches.

\subsection{Review of $I[\kappa]$}\label{S2.3}
 
We review the approachability ideal $I[\kappa]$ and its variant $\checki[\kappa]$ (from \cite[Definition 6, page 360, page 377]{Sh:108})   that is suitable also for the description of regular limit cardinals $\kappa$. Our review focuses on results that we use in Section~\ref{S6}.

 \begin{definition}[The Approachability Ideal on Successors \cite{Sh:88a}]\label{approachability_ideal}
 Let \\$\bar{a}=\seq{a_\alpha}{\alpha< \kappa}$ enumerate  a subset of $\kappa^{<\kappa}$.
  The ideal $I[\kappa](\bar{a})$ is be the set of
$ S \subseteq \kappa$  such that for a club $C \subseteq \kappa$ for any
 $\delta \in S \cap C$, there is a set  $A_\delta \subseteq \delta$ that is cofinal in $\delta$  with   $\ot(A_\delta) = \cf(\delta)<\delta$ and satisfies $\{A_\delta \cap \beta : \beta < \delta\} \subseteq \{a_\alpha: \alpha < \delta\}$. 
 The \emph{approachability ideal} $I[\kappa]$ is the union of all the $I[\lambda](\bar{a})$, $\bar{a}$ as above. If $\kappa^{<\kappa} = \kappa$, we let $\seq{a_\alpha}{\alpha< \kappa}$ be an enumeration of $\kappa^{<\kappa}$ and set
 $I[\kappa] = I[\kappa](\bar{a})$.
 \end{definition}
 
\begin{remark} Equivalently we can  require in addition that the $A_\delta$ be closed.
 The reason is, that we can choose $\bar{a}$ so that if there is a sequence of unbounded witnesses $\seq{A_\delta}{\delta \in S}$ for $S \in I[\kappa](\bar{a})$ then there is also a sequence of club witnesses $\seq{C_\delta}{\delta \in S}$ for $S \in I[\kappa](\bar{b})$ for a slightly richer sequence $\bar{b} \in {}^\kappa ([\kappa]^{<\kappa})$. For a detailed proof we refer to \cite[Lemma 4.4]{Sh:351}.
 \end{remark}

Most of the literature on $I[\kappa]$ in \cite{Sh:88a}, \cite{Sh:108}, \cite{Eisworth_handbook} focusses on the case of $\kappa$ being a successor cardinal. For a successor cardinal $\kappa$, the regular cardinals below $\kappa$ form a non-stationary set, and dropping the clause $\ot(C_\delta) = \cf(\delta) < \delta$
\autoref{approachability_ideal} yields an equivalent notion. We work with a version that dispenses with $\cf(\delta) < \delta$.

 \begin{definition}[From \cite{Sh:88a}, {\cite[Definition 6 and page 377]{Sh:108}}, \cite{Sh:351}]\label{approachability_ideal_original}
 Let \\$\bar{a}=\seq{a_\alpha}{\alpha< \kappa}$ enumerate  a subset of $\kappa^{<\kappa}$.
  The ideal $\checki[\kappa](\bar{a})$ is be the set of
$ S \subseteq \kappa$  such that for a club $C \subseteq \kappa$ for any
 $\delta \in S \cap C$, there is a club  $C_\delta \subseteq \delta$ that is cofinal in $\delta$  with   $\ot(C_\delta) = \cf(\delta)$ and satisfies $\{C_\delta \cap \beta : \beta < \delta\} \subseteq \{a_\alpha: \alpha < \delta\}$. 
 The \emph{approachability ideal} $\checki[\kappa]$ is the union of all the $\checki[\lambda](\bar{a})$, $\bar{a}$ as above. If $\kappa^{<\kappa} = \kappa$, we let $\seq{a_\alpha}{\alpha< \kappa}$ be an enumeration of $\kappa^{<\kappa}$ and have
 $\checki[\kappa] = \checki[\kappa](\bar{a})$.
 \end{definition}

 Note, that $\{\delta < \kappa : \delta \mbox{ is regular}\} \in \checki[\kappa]$.
 We just take $\bar{a} = \{\alpha : \alpha \in \kappa\}$ and for regular $\delta < \kappa$, $C_\delta = \delta$.
 
Many authors call $\checki[\kappa]$ now  $I[\kappa]$, see e.g. \cite{Mitchell_Ilambda}, \cite{eightfold}, \cite{Krueger_2019}.
 Many equivalent definitions of the ideal are given in \cite{Sh:351} and \cite{Eisworth_handbook}.

\begin{definition} \label{kappaapproxseq} Let $\kappa$ be a regular cardinal. 
A $\kappa$-approximating sequence 
$\mathfrak M= \seq{M_i}{i < \kappa}$ is a continuously increasing sequence of $M_i \prec (H(\chi), \in , <^*_\chi, \kappa, \dots )$ for some regular cardinal $\chi > 2^{2^\kappa}$ with $\seq{M_i}{i<j} \in M_{j+1}$ and 
$|M_i| < \kappa$ and a countable signature.   We define $S[\mathfrak M]$ to be the set of $\delta < \kappa$ such that $\delta$, $M_\delta \cap \kappa = \delta$ and there is a cofinal subset $A$ of $ \delta$ of order-type $\cf(\delta)$ with the property that every proper initial segment of $A$ is in $M_\delta$. \end{definition}

The set $S(\mathfrak M)$ is called $S_2(\mathfrak M)$ in \cite[Definition 6]{Sh:108}.

\begin{theorem}[Shelah, Eisworth {\cite[Theorem 3.6]{Eisworth_handbook}}] $S \in \checki[\kappa]$ iff there is a club $E $ in $\kappa$, such that $S \cap E \subseteq S(\mathfrak M)$ for a $\kappa$-approximating sequence $\mathfrak M$.
\end{theorem}

\begin{definition} For a subset $E$ of $\kappa$, let $\acc(E) = E \cap \{\alpha< \kappa: \alpha = \sup(E \cap \alpha)\}$,
 and $\nacc(E) = E \setminus \acc(E)$.
 \end{definition}

\begin{theorem}[Shelah \cite{Sh:88a}, Eisworth {\cite[Theorem 3.7]{Eisworth_handbook}}]\label{club_nacc_charact.} Let $\kappa$ be a regular uncountable cardinal. For any $S \subseteq \kappa$ the following are equivalent:
\begin{enumerate}
\item[(1)] $S \in I[\kappa]$,
\item[(2)] There is a sequence $\seq{C_\alpha}{\alpha < \kappa}$ and a closed unbounded $E \subseteq \kappa$ such that for any $\alpha < \kappa$,
\begin{enumerate}
\item[(a)] $C_\alpha$ is a closed but not necessarily unbounded subset of $\alpha$, 
\item[(b)] for any $\beta \in \nacc(C_\alpha)$ we have $C_\beta = C_\alpha \cap \beta$, 
\item[(c)] if $\alpha \in E \cap S$ then $\alpha$ is singular and $C_\alpha$ is a closed unbounded subset of $\alpha$ of order-type $\cf(\alpha)$.
\end{enumerate}
\end{enumerate}
\end{theorem}

\nothing{We let for a regular cardinal $\lambda < \kappa$, 
$E_\lambda^\kappa = \{\alpha \in \kappa: \cf(\alpha) = \lambda\}$. In other places this is called $\kappa \cap \cof(\lambda)$.
We review some statements on the extent of $\checki[\lambda]$. For our applications.

\begin{theorem}[Shelah,\cite{Sh:88a} Eisworth]\label{four_cardinals} Suppose $\lambda^+ < \sigma < \kappa$ for regular cardinals $\lambda<  \sigma< \kappa$. There is a set $S \subseteq E^\kappa_\lambda \in  I[\kappa]$ such that $S \cap \theta$ is stationary in $\theta$ for stationarily many $\theta$ in $E_\sigma^\kappa$. In particular $I[\kappa]$ contains a stationary subset of $E^\kappa_\lambda$.
\end{theorem}

Note that $\kappa^{<\kappa} = \kappa$ implies for $\kappa = \lambda^+$ that $\lambda$ is a singular strong limit.

 \begin{theorem}[Shelah \cite{Sh:88a}, Eisworth ] Suppose $\kappa= \lambda^+$ for a strong limit singular
 $\lambda$. If $\mu < \lambda$ is a regular cardinal, then $I[\kappa]$ contains a stationary subset of $\kappa \cap \cof(\mu)$.
\end{theorem}

We remark that the following theorem requires just a one cardinal gap between $\lambda$ and $\kappa$, however, it uses an hypothesis on cardinal exponentiation.

 \begin{theorem}[Shelah \cite{Sh:88a}, Eisworth {\cite[Theorem 3.6]{Eisworth_handbook}}]\label{successor_exponent}
 let $\lambda < \kappa$ be regular cardinals such that $\eps^{<\lambda} < \kappa$ for any $\eps < \kappa$. Then $E^{\kappa}_{\lambda} \in \checki[\kappa]$.
 \end{theorem}
 }

The following theorem shows that at any uncountable $\kappa = \kappa^{<\kappa}$ the premises of \autoref{l5} are fulfilled for suitable $\lambda$. 
 
 \begin{theorem}[Shelah \cite{Sh:88a}, \cite{Sh:351}]\label{successor_exponent_2}
 let $\lambda < \kappa$ be  cardinals  such that $\kappa$ is regular and  $\kappa^{<\lambda} \leq \kappa$. Then there is a stationary set
 $S \subseteq \kappa \cap \cof(\cf(\lambda))$ with $S \in \checki[\kappa]$.
 \end{theorem}
 \begin{proof}
We let $\seq{a_\alpha}{\alpha < \kappa}$ enumerate ${}^{\lambda >}  \kappa$, such that each element appears $\kappa$ often.
We let 
\begin{equation*}\begin{split}
S = & \{\delta \in \kappa \cap \cof(\cf(\lambda)): (\exists \eta \in {}^{\cf(\lambda)} \delta)\\
& \bigl(\sup(\range(\eta)) = \delta \wedge (\forall i < \cf(\lambda))( \exists j < \delta)  \eta \rest i = a_j\bigr)\}.
\end{split}
\end{equation*}
By definition, $S \in \checki[\kappa](\bar{a})$. We show that $S$ is stationary.
Let $C \subseteq \kappa$ be a club. By induction on $i < \cf(\lambda)$ we choose $\eta_i \in {}^{i+1} \kappa$ and $\delta_i \in \kappa$ with the following properties for any $i < \cf(\lambda)$,
\begin{enumerate}
\item[(a)] $\delta _i \in C$
\item[(b)] for $i < j < \cf(\lambda)$, $\delta_i < \delta_j$,
\item[(c)] $\eta_i = \seq{\delta_j}{j\leq i}$,
\item[(d)] there is $k \leq \delta_{i+1}$ with $\eta_i = a_k$.
\end{enumerate}

$i = 0$: We let $\delta_0 \in C$. We let $\eta_0 = \{(0, \delta_0)\}$.

Successor step: $i = j+1$. We choose $\delta_i \in C \setminus (\delta_j+1)$ such that there is some $k \leq \delta_{j+1}$ with  $\eta_j = \seq{\delta_\ell}{\ell \leq j} = a_k$. Then we let $\eta_{i} = \eta_j \cup\{(i,\delta_i)\}$. 

Limit step $i< \cf(\lambda)$: We let $\delta_i = \sup\{\delta_j: j < i\}$ and  $\eta_i = \bigcup\{\eta_j : j < i\} \cup \{(i, \delta_i)\}$.
Then we pick $\delta_{i+1}$ such that for some $k \leq \delta_{i+1}$,
$\eta_i = a_k$.

Then $\eta = \bigcup\{\eta_i : i < \cf(\lambda)\}$ witnesses $\delta= \sup\{\delta_i : i < \cf(\lambda)\} \in S \cap C$.
\end{proof}

   \subsection{Forcing}
   Our notions of forcing are written in Israeli style: $p \leq q$ means that $q$ is stronger than $p$. We write $\bP \Vdash \varphi$ if any condition in $\bP$ forcing $\varphi$. Equivalently on can say the weakest condition of $\bP$ forces $\varphi$.

   \section{The Case of $\kappa$ being Weakly Mahlo}\label{S3}
   
  We consider regular limit cardinals $\kappa$ that are not necessarily strong limits.  For $\kappa$ being weakly Mahlo there is an two step derivation of a name of a diamond, which we present in this section.   
We show that diamond in the one-step-extension leads to \autoref{cor2}.

   \begin{definition}
   	Let $\delta$ be an ordinal of unbountable cofinality. Let $S \subseteq \delta$ be  stationary in $\delta$. 
   	\begin{enumerate}
   		\item[(1)]
   		The quantifier $\forall^{\rm club}\alpha \in S, \varphi(\alpha)$ says that
   		there is a club $C $ in $\delta$ such that $S_{\varphi}= \{\alpha \in S :\varphi(\alpha)\} \supseteq S \cap C$.
   		\item[(2)]
   		We define 
   		the quantifier $\exists^{\rm stat}\alpha \in S, \varphi(\alpha)$ as
   		$S_{\varphi}= \{\alpha \in S :\varphi(\alpha)\}$ is a stationary subset of $\delta$.
   \end{enumerate}\end{definition}

\begin{definition}\label{name_eta}  Let $G$ be a $\bP$-generic filter over $V$ and assume that $\bP$ is one of our named forcings. The following function $\eta \colon \kappa \to \kappa$ is called \emph{the generic branch}:
$\eta = \bigcup\{ \stem(p) : p \in G\}$. 
We let $\name{\eta}$ be name for $\eta$.
\end{definition}

We name a combinatorial principle $\boxplus_{\kappa, S}$.
This asserts that there are stationarily many $\delta \in S$ for which $\delta$ can be partitioned into $\delta$-many parts such that each of them is stationary in $\delta$,  via a partition that does not depend on $\delta$.
   
   \begin{definition}\label{boxplus_easy} 
   	Let $\kappa$ be a weakly Mahlo cardinal and let \\$S \subseteq \{\delta \in \kappa  \such \delta \mbox{ is a regular limit cardinal}\}$. 
   	\begin{enumerate}
   		\item[$\boxplus_{\kappa, S}$] is the following statement: 
   		There is a function $f \colon \kappa \to \kappa$ such that
   		for any $\alpha < \kappa$, $f(\alpha) < \min(\alpha,1)$ and
   		\begin{equation}\label{code3}
   			\begin{split}
   				& (\exists^{\rm stat} \delta \in S)
   				(\forall \beta < \delta)\\
   				&
   				 (S_{\delta,\beta} := \{\gamma \in  \delta : f(\gamma) = \beta\} \mbox{ is stationary in }\delta).
   			\end{split}
   		\end{equation}
   		\end{enumerate}
  
   \end{definition}

   Now the proof of \autoref{main_theorem_easy_case} consists of \autoref{boxplus1} and \autoref{f8}.
   
   \begin{lemma}\label{boxplus1} If $\kappa$ is weakly Mahlo and $S \subseteq \{\delta < \kappa : \delta \mbox{ is a regular limit}\\ \mbox{cardinal}\}$ then $\boxplus_{\kappa, S}$.
   \end{lemma}
   
   \begin{proof} We let
   \begin{equation}\label{the_f2}
   	f(\gamma) = \begin{cases} \beta, & \mbox{ if } \cf(\gamma) = \aleph_{\beta+1};\\
0, & \mbox{ else.}
\end{cases}		
   		\end{equation}  
   		and  Statement~\eqref{code3} holds in the slightly stronger form 
   \begin{equation*}\label{code3sharp}
   			\begin{split}
   				& (\forall \delta \in S)
   				(\forall \beta < \delta)
   				 ( \{\gamma < \delta : f(\gamma) = \beta\} \mbox{ is stationary in }\delta).
   			\end{split}
   		\end{equation*}		
   		
   \end{proof}
   
   \begin{definition}\label{acc} For $E \subseteq \kappa$ we write $\acc^+(E)= \{\alpha\in \kappa \such 
   	\alpha = \sup(E \cap \alpha)\}$ and $\acc(E) = E \cap \acc^+(E)$.
   \end{definition}
 \nothing{  \begin{definition} Let $\bP$ be forcing notion in which stronger condition are just
   	subsets, i.e., for $p, q \in \bP$, $q \geq p$ if $q \subseteq p$. We say an ascending (in strength, Jerusalem notation) sequence $\seq{p_\eps}{\eps \leq \delta}$ of conditions is \emph{continuous}, if at any limit ordinal $\alpha \leq \delta$ we have $p_\alpha = \bigcap\{p_\beta : \beta < \alpha\}$.
   \end{definition}

First we need to recall a definition.
\begin{definition}\label{approx_seq_over}
	Let $\tau$ be at most countable and $\theta > \kappa$ be a regular cardinal. A sequence $\bar{M}= \seq{M_i}{i < \kappa}$ is called \emph{a $\kappa$-approximating sequence over $x$ (for the signature $\tau$)} if 
	\begin{enumerate}
		\item[(1)]
		$x \in H(\theta)$, $\theta > \kappa$ sufficiently large,
		\item[(2)] for each $i < \kappa$, $M_i \prec (H(\theta), \in, <_\theta, \{x\}, (Z)_{Z \in \tau} )$ (where $<_\theta$ denotes a well order on $H(\theta)$),
		\item[(3)]  for any $i<\kappa$, $\seq{M_j}{j<i} \in M_{i+1}$,
		\item[(4)]  for any $i<\kappa$,
		$|M_i| < \kappa$,
		\item[(5)] for limit ordinals $i \leq \kappa$, $M_i = \bigcup \{M_\alpha: \alpha < i\}$. 
	\end{enumerate}
\end{definition}
By the L\"owenheim-Skolem theorem, for any regular $\kappa$ and any set $x$ and any countable $\tau$ there is a $\kappa$-approximating sequence over $x$.
}

We state the following lemma for Sacks forcing $\qb$. It  holds for any of the four types of tree forcings. For Miller forcing and for Laver forcing, we  work with one fixed partition of $\kappa$ into two stationary sets $T_0$, $T_1$. This partition is used to define the trunk lengthenings: For $j = 0,1$, $\name{\eta}(\eps) = j$ in Equation \eqref{diamondname3}, in Clause $(\ast)_3$(e), and in Equations \eqref{aste}, \eqref{arranged1} is replaced by $\name{\eta}(\eps) \in T_j$.

   \begin{lemma}\label{f8} Let $\kappa > \aleph_0$ and $S \subseteq \kappa$ be stationary. If $\boxplus_{\kappa, S}$ holds, then
   	$\qb \Vdash \diamondsuit_\kappa(S)$.
   \end{lemma}
   
   \begin{proof} 
We let $\boxplus_{\kappa, S}$ witnessed by $f$ and let for $\delta \in S$, 
$S_{\delta, \beta} = \{\eps < \delta : f(\eps) = \beta\}$. For stationarily many $\delta \in S$, for any $\beta<\delta$, $S_{\delta, \beta}$ a stationary subset of $\delta$. Let $S'$ be a stationary set of these good $\delta$.  We define the name $\langle \name{\nu}_\delta : \delta \in S \rangle$ for a sequence by letting for $\delta \in S'$, $\beta < \delta$, $j = 0,1$,
   		\begin{equation} \label{diamondname3}
   			\begin{split}
   			\qb \Vdash	& \mbox{``}\name{\nu}_\delta(\beta) = j \; \leftrightarrow \;   	(\forall^{\rm club} \eps \in S_{\delta, \beta}) (\name{\eta}(\eps) = j)\mbox{''}.
   		\end{split}\end{equation}
   For $\delta \in S \setminus S'$, we can let $\name{\nu}_\delta$ be a name for the zero sequence of length $\delta$.

   	\begin{equation*}%\tag{$\ast_1$}
   		\qb \Vdash  \mbox{``}\seq{\name{\nu}_\delta}{\delta \in S} \mbox{ is a $\diamondsuit_\kappa(S)$-sequence.''} 
   	\end{equation*}
   	Towards this suppose that 
   	\begin{equation*}
   		p \Vdash \mbox{``}\name{\xname} \in {}^\kappa 2, \mbox{ and } \name{D} \mbox{ is a club subset of }\kappa.\mbox{''}
   	\end{equation*}

   	We show that there is some $q \geq p$ that forces $\delta \in \name{D}\cap S'$ and $\name{x} \rest \delta = \name{\nu}_\delta$.

We let $\chi = \beth_\omega(\kappa)$ and let $<^*_{\chi}$ be a well-ordering of $H(\chi)$. We choose a $\kappa$-approximating (see \autoref{kappaapproxseq}) sequence 
$\seq{N_\eps}{\eps < \kappa}$ with 
\begin{equation*}%\tag*{\quad $(\ast)_1$}\label{ast1}
{\bf c} = (\kappa, p, \name{\bar{\nu}}, \name{x}, \name{D}, S) \in N_0.
\end{equation*}
We let $E = \{\alpha < \kappa: N_\alpha \cap \kappa = \alpha \}$.
We pick any $\delta$ with 
\begin{equation*}%\tag*{\quad $(\ast)_2$}
 \delta \in S' \cap \acc(E) 
\end{equation*}
For $\eps < \delta$, we let $\kappa_\eps = N_\eps \cap \kappa$.

By induction on $\eps \leq \delta$
we chose a condition $p_\eps$ such that for any $\eps$ the following holds
\begin{enumerate}\item[$(\circledast)$] $p_\eps$ is the $<^*_\chi$-least element such that 
\begin{enumerate} 
\item[(a)] $p_\eps \geq p$. (We shall prove that $p_\eps \in N_{\eps+1}$ in our construction but it is better  to not state it in our demands.)
\item [(b)] $p_\eps \geq p_\zeta$ for $\zeta<\eps$. 
\item[(c)] $\lg(\tr(p_\eps)) \geq \kappa_\eps$.
\item[(d)] $p_\eps$ forces values to $\name{\xname} \rest \kappa_\eps$, $\name{D} \cap \kappa_\eps$ and $\min (\name{D} \setminus \kappa_\eps)$ call them $\xname_\eps$, $e_\eps$, $\gamma_\eps$ respectively.
\item[(e)] 
For limit ordinals $\eps< \delta$, $\tr(p_\eps)(\kappa_\eps) = \xname_\eps(f(\kappa_\eps))$. 
\end{enumerate}
\end{enumerate}
We can carry the induction since $\bQ$ is $(<\kappa)$-complete and for clause (e) we recall $f(\kappa_\eps)< \kappa_\eps$. More fully, let us prove $(\circledast_\eps)$ by induction on $\eps$, 

\begin{enumerate}
\item[$\circledast_\eps$]
\begin{itemize}
\item[•] ${\bf m}_\eps = \seq{p_\zeta, \xname_\zeta, e_\zeta, \gamma_\zeta}{\zeta<\eps}$ exists and is unique and $\zeta<\eps$ implies ${\bf m}_{\zeta+1} \in N_{\zeta+1}$.

\item[•] ${\bf m}_\eps$ is defined in $(H(\chi), \in , <_\chi^*)$ by a formula $\varphi = \varphi(x, \bar{y})$ with $x $ for ${\bf m}_\eps$ and $\bar{y} = (y_0, y_1)$ with $y_0 = \bar{N} \rest \eps$ and $y_1 = {\bf c}$ from $(\circledast)_1$(e). 
\end{itemize}
\end{enumerate}
\nothing{Moreover, our inductive choice fulfils at any $\eps< \delta$,
\begin{equation}\tag*{\quad $(\ast)_4$}
{\bf m}_\eps \in N_{\eps+1}.
\end{equation}
}

{\bf Case 1} $\eps = 0$. 

${\bf m}_0 = \langle \rangle$  and $\lg(\bar{N} \rest 0) = 0$.

{\bf Case 2} $\eps = \zeta+1$.

  Now $p_\eps$ is the $<_\chi^*$ first element of $\bQ$ satisfying clauses $(\circledast)_3$(a) - (d). There is no requirement (e), since $\eps$ is a successor.  Clearly such a $p$ exists and hence one of them must be $<_\chi^*$-least. As each element of $H(\chi)$ mentioned above is computable from $\bar{N} \rest \eps$, it belongs to $N_\eps = N_{\zeta+1}$ since $\bar{N} \rest (\zeta+1) \in N_{\zeta+1}$ by $(\circledast)_1$(d)

{\bf Case 3} $\eps$ limit.\\
This is the only place at which we use the specific choice of $\bQ = \qb$ and not just $(<\kappa)$-complete forcings that force $\name{\eta} \not\in\bV$ but that the limit of splitting nodes is a splitting node.
By $(<\kappa)$-completeness $q = \bigcap_{\zeta<\eps} p_\zeta$ in $\qb$. Also for $\zeta < \eps$, $\lg(\tr(p_\zeta)) \geq \kappa_\zeta = N_\zeta \cap \kappa \in N_{\zeta+1}$. Hence $\lg(\tr(p_\zeta)) \in [\kappa_\zeta, \kappa_{\zeta+1})$  for $\zeta <\eps$. Also for 
\[
\forall \zeta\leq \xi <\eps (\tr(p_\xi) \in \Split(p_\zeta))
\] 
by induction hypothesis. Hence by the definition of the Sacks forcing \autoref{qb} clause number (2) we have
$\forall \zeta < \eps (\bigcup \{\tr(p_\xi) : \xi < \eps\} \in \Split(p_\zeta)$ 
and $\bigcup \{\tr(p_\xi) : \xi < \eps\} \in \Split(q)$.
We have $\lg(\tr(q)) = \kappa_\eps$, since $\bar{N}$ is continuous and $\kappa_\eps = N_\eps \cap \kappa$.
Moreover 
$q \Vdash x_{\kappa_\eps} = \name{x} \rest \kappa_\eps$.
We can compute $\tr(q)$, $\kappa_\eps$ and $f(\kappa_\eps)$ from $(\bar{N} \rest \eps, {\bf c})$.
Moreover 
\begin{equation}\label{gut1}q \Vdash x_{\kappa_\eps} = \name{\xname} \rest \kappa_\eps \wedge \kappa_\eps \in \name{D}.
\end{equation}
by the induction hypothesis $(\circledast)$(a) -- (d) and since $\eps$ is a limit.

Hence $\tr(q)$ has two immediate successors of length $\kappa_\eps +1$ and we can let $p_\eps \geq q$ and 
\begin{equation}\label{aste}
\tr(p_\eps)(\kappa_\eps) = \xname_{\kappa_\eps}(f(\kappa_\eps)).
\end{equation}

 This is clause $(\circledast)$(e) that we have to fulfil.

Now we carried the induction. We let $q = \bigcap_{\eps < \delta} p_\eps$.
We show that 
\begin{equation}\label{diam1}
q \Vdash \name{\xname} \rest \delta= \name{\nu}_\delta.
\end{equation}
Equation \eqref{gut1} implies at the limit $\delta$:  $q \Vdash \name{\xname} \rest \delta= \bigcup\{\xname_\eps: \eps < \delta\}$.

We fix $\beta < \delta$. We verify that  for club many $\eps$ in the stationary set $S_{\delta, \beta}$ we have
\begin{equation}\label{arranged1}
q \Vdash (\kappa_\eps = \eps \wedge f(\eps) = \beta) \rightarrow \eta(\eps) = \name{\xname}(\beta) = \xname_\eps(\beta).
\end{equation}
This follows from Equations~\eqref{gut1} and \eqref{aste} that we made true at club many $\kappa_\eps$, and thus as club many $\eps$, since $\eps  \mapsto \kappa_\eps$ is a continuously increasing function on $\eps < \delta$ and $\kappa_\delta = \delta$.

   \end{proof}
   
%  This concludes the proof of \autoref{main_theorem_easy_case}.

\medskip

   Now  we turn to \autoref{cor2}.\\
      We notice that \autoref{claim2.19} and \autoref{claim2.21} hold also for any of our forcings. They could be mixed along an iteration.   We call the first iterand $\bP_1$.
   
   \begin{lemma}\label{claim2.19} Let $\bP$ be a $\leq \kappa$-supported iteration of iterands of $\qb$. For proving \autoref{cor2}, it suffices to prove
  	$\bP_1 \Vdash \diamondsuit_\kappa$ and that the forcing $\bP_1$ does not collapse $\kappa^+$. 
  \end{lemma}
  
 \begin{proof} If $\kappa> \aleph_1$ is a successor cardinal, \cite{Sh:922} gives the diamond in $\bV$. Now let $\kappa$ be a regular limit cardinal. Let $\bG$ be $\bP_1$ generic over $\bV$. In $\bV[\bG]$ we apply \autoref{kanamori_theorem} to the $(\leq \kappa)$-support iteration
  	$\langle \bP_\alpha/\bG, \name{\bQ}_\beta/\bG : \alpha \in [1, \delta], \beta \in [1, \delta)\rangle$. 
  \end{proof}
  
For defining fusion sequences, we use a notion that is suitable for $p \subseteq {}^{\kappa>}{\kappa}$ and which could be simplified for $p \subseteq{}^{\kappa>} 2$.
  
   \begin{definition}\label{axiomA}
   We assume $\kappa = \kappa^{<\kappa}$. We conceive a forcing notion as a tree $p \subseteq \kappa^{<\kappa}$ or $\subseteq 2^{<\kappa}$. 
   	Recall, splitting means splitting into a club.  For $\alpha < \kappa$ we let 
   	\[\spl_\alpha(p) = \bigl\{ t \in \Split(p)
   	\such \ot(\{s \subsetneq t \such s \in \Split(p)\}) = \alpha\bigr\}
   	\] and with a fixed enumeration $\{\eta_\alpha : \alpha < \kappa\} $ of ${}^{\kappa> } \kappa$ we define
   	\[\cl_\alpha(p) := \{s \in p \such (\exists \gamma \leq \alpha)(\exists t \in \spl_\gamma(p))( s \subseteq t) \wedge
   	(\exists \beta \leq \alpha)( s = \eta_\beta)\}.
   	\]
   	We let  $p \leq_ \alpha q$
   	if  $p \leq q$ and $\cl_\alpha(p) = \cl_\alpha(q)$.
   \end{definition}
   
   \begin{lemma}\label{claim2.21} Under $\kappa^{<\kappa} = \kappa$, the forcing $\bP_1$ is $\kappa$-proper and does not collapse $\kappa^+$.
   \end{lemma}
   \begin{proof}
   	Let $\chi> 2^\kappa$ is a regular cardinal.
   	Let $p \in \qa$ and let $\name{\tau}$ be a name for function from $\kappa$ into $\kappa^+$. Here the cardinal successor is interpreted in the ground model.
 We pick an $N \prec H(\theta)$ of size $\kappa$ with ${}^{<\kappa} N $ with $\kappa$, $p$, $\bP_1 \in N$ and let $\langle I_\eps : \eps < \kappa \rangle $ list all dense subseteq of $\bP_1$ that are elements of $N$.  	
   	Now by induction on $\eps< \kappa$ we choose conditions $p_\eps$, and sets $\{a_{s \concat \langle i \rangle} : s \in \spl_\eps(p_\eps), i \in \osucc_{p_\eps}(s)\} \subseteq {}^{(\eps +1)} \kappa$  with the following properties:
   	\begin{enumerate}
   \item[(a)] $p_\eps \in N$.	
   		\item[(b)] $p_0 = p$.
   		\item[(c)] If $\eps < \delta$, then $p_\eps \leq_\eps p_{\delta}$.
   		\item[(d)]  At limits $\eps$, $p_\eps = \bigcap \{p_\delta: \delta < \eps\}$.
   		\item[(e)] if $s \in \cl_\eps(p_\eps) \cap \spl_\eps(p_\eps)$, then for every $i \in \osucc_{p_\eps}(s)$, the condition $p_{\eps+1}^{\langle s \concat \langle i \rangle\rangle} $ is in $I_\eps$ and forces
   		$\name{\tau} \rest (\eps+1) = a_{s \concat \la i \ra}$. 
   	\end{enumerate}
   	In the end the fusion $q = \bigcap\{p_\eps: \eps < \kappa\} = \bigcup\{\cl_\eps(p_\eps) : \eps < \kappa\}$ is am $N$-generic condition, since it forces for any $\eps < \kappa$ that one of the $q^{\langle s \concat \langle i \rangle\rangle}$, $s \in 
   	\cl_\eps(q) \cap \spl_\eps(q)= \cl_\eps(p_\eps) \cap \spl_\eps(p_\eps)$, $i \in \osucc_{p_\eps}(s)$, is in $\bG \cap I_\eps \cap N$. For each $\eps < \kappa$, we have 
   	for any $s \in \cl_\eps(q) \cap \spl_\eps(q)= \cl_\eps(p_\eps) \cap \spl_\eps(p_{\eps})$, $i \in \osucc_{p_\eps}(s)$, $q^{\langle s \concat \langle i \rangle\rangle} \geq p_{\eps+1}^{\langle s \concat \langle i \rangle\rangle} $. The condition $p_{\eps+1}$ forces that $\name{\tau}\rest(\eps+1)$ is one of the  values in
   	\[
   	K_\eps= \{a_{s \concat \la i \ra} : s \in \spl_\eps(q)= \spl_\eps(p_{\eps}), i \in \osucc_{p_{\eps}}(s)\}
   	\] 
   	that are given by  $p_{\eps+1}^{\langle s \concat \langle i \rangle\rangle} $, $s \in \spl_\eps(q)= \spl_\eps(p_{\eps})$, $i \in \osucc_{p_{\eps}}(s)$.
   	For any $\eps$, the stronger condition $q$ forces this. By $\kappa^{< \kappa} = \kappa$, we have $|K_\eps| \leq \kappa$.
   	Since $|\bigcup\{K_\eps : \eps < \kappa\} | \leq \kappa$, the forcing
   	preserves $\kappa^+$ as a cardinal.
   \end{proof}
   
   This concludes the proof of \autoref{cor2} in the weakly Mahlo case. In the complementary case, we finish with \autoref{b14}.

   \begin{remark}
In the tree forcings considered here, any stationary subset $S$ of $\kappa$ stays stationary in any $\bP_1$-extension. This is so since $\bP_1$ is strongly $(<\kappa)$-distributive, i.e.,  for any sequence $\seq{D_\beta}{\beta<\kappa}$ of dense subsets of $\qb$ and any $p \in \bP_1$, there is a sequence $\seq{p_\beta}{\beta<\kappa}$ such that for $\beta < \kappa$, $p_\beta \in D_\beta$ and $p_0 \leq p$, see \cite[Lemma 3.8]{Hannes_Jakob_23}.

   	\nothing{Every stationary set in $\checki[\kappa]$ is preserved by any $< \kappa$-closed forcing by \cite{Sh:88a}, see \cite[Theorem 2.18, Theorem 3.6]{Eisworth_handbook}. The forcing orders $\qb$ and $\qa$ are $(<\kappa)$-closed.}
   \end{remark}

 \section{The Case of a Regular Limit Cardinal $\kappa$}
 \label{S4}

In this section we prove \autoref{b14}.
% Now we look at the regular limit cardinals that are no weakly Mahlo.
Let $\kappa$ be a regular limit cardinal. If the set of regular cardinals below $\kappa$ is stationary, then \autoref{main_theorem_easy_case} applies. If not, the set  $S_{{\rm sing}, \kappa}$ of singular cardinals in $\kappa$ contains a club in $\kappa$. 
 In any case, $S_{{\rm sing}, \kappa}$ is stationary in $\kappa$. We apply the 
regressive function $\cf \colon S_{{\rm sing},\kappa} \to \kappa$  and find a regular cardinal $\mu$ and such that $S_{{\rm sing},\kappa} \cap \cof(\mu)$ is stationary in $\kappa$.
Also for a weakly Mahlo cardinal, the set $S_{{\rm sing},\kappa}$ of singular cardinals is stationary, and hence for some $\mu$, also $S_{{\rm sing}, \kappa}\cap \cof(\mu)$ is stationary in $\kappa$. So for the main iterability theorem, we can do without \autoref{main_theorem_easy_case}.

First we recall club guessing.

\begin{theorem}[{\cite[Def. III.1.3, Claim III.2.7, page 128]{Sh:g}}]\label{club_guessing}
Let $\mu  < \kappa$, $\kappa$ be a regular limit cardinal, and $\cf(\mu) = \mu$.
Let $S \subseteq \kappa \cap \cof(\mu)$ be stationary in $\kappa$. There is a sequence $\seq{C_\delta}{\delta \in S}$ with the following properties:
For any
club $E$ there is some $\delta \in E\cap S$ such that $C_\delta \subseteq E \cap \delta$. Moreover, for club many $\delta \in S$,  $\sup\{\cf(\alpha): \alpha \in C_\delta\} = |\delta|$.
\end{theorem}

So we can thin out $S$ to such a club $C$ as in the last sentence and for each 
$\delta \in S \cap C$, we choose a cofinal club subsequence $\seq{\alpha_{\delta, i}}{i < \mu}$ in $C_\delta$ such that  
$\lim_{i<\mu} \seq{\cf(\alpha_{\delta, i+1})}{i<\mu} = |\delta|$.

\begin{fact}
For any weakly inaccessible $\kappa$, $\aleph_\kappa = \kappa$ and $C_{{\rm fix}, \kappa}= \{\alpha , \kappa: \aleph_\alpha = \alpha \}$ is club in $\kappa$.
\end{fact}
\begin{proof}
First $\kappa$ is a limit cardinal, so there is a limit ordinal $\lambda$, such that $\kappa = \aleph_\lambda$. In addition $\kappa$ is regular, hence $\kappa = \cf(\kappa) = \cf(\lambda)$. So we have $\lambda = \kappa$.
Since the $\aleph$-operation is continuous, the set $C_{{\rm fix}, \kappa}$ is closed. We show that it is unbounded. To this end, let $\alpha < \kappa$. Then $\aleph_\alpha < \aleph_\kappa = \kappa$. Hence we can define $\alpha_0 = \alpha$, $\alpha_{n+1} = \aleph_{\alpha_n}$ for $n < \omega$. Then $\bigcup \{\alpha_n : n< \omega \} \in C_{{\rm fix}, \kappa}$.
\end{proof}

If $S \subseteq C_{{\rm fix }, \kappa}$, then the chosen sequences from above fulfil
$\lim_{i<\mu} \seq{\cf(\alpha_{\delta, i+1})}{i<\mu} = \delta$.  Such sequences will be important below.

We start with a combinatorial principle (in ZFC) and a name for a possible diamond sequence.
\begin{definition} We let $\forall^{\rm club} i < \mu, \varphi(i)$ mean 
\begin{itemize}
\item[•] for any large enough $i < \omega$, $\varphi(i)$, if $\mu = \omega$,
\item[•] $\{i < \mu :\varphi(i)\}$ contains a club, if $\mu > \omega$.\end{itemize}
 We let $\exists^{\rm stat} \alpha < \delta, \varphi(\alpha)$ mean that the 
 set of $\alpha$ with $\varphi(\alpha)$ is stationary in $\delta$.
\end{definition}

\begin{definition}\label{b2}
  Let $\kappa > \mu$, $\kappa$ be a regular limit cardinal, $\mu$ regular cardinals, and let $S \subseteq \kappa \cap \cof(\mu)$ be stationary in $\kappa$.
  Also $\mu= \omega$ is possible.
\begin{enumerate}
\item[ $\boxplus_{\kappa, \mu, S}$] is the following statement: There are $\bar{C}$, $\bar{\alpha}$ and $f$ with the following properties:
 \begin{enumerate}
 \item[(a)]   
 $\bar{C}=\seq{C_\delta}{\delta \in S}$. $C_\delta$ is a club in $\delta$ of order type $\mu$.
 \item[(b)] For each $\delta \in S$, the sequence $\seq{\alpha_{\delta, i}}{i < \mu}$ is an increasing continuous enumeration of $C_\delta$. We write
\[
\bar{\alpha} = \seq{\seq{\alpha_{\delta, i}}{i<\mu}}{\delta \in S}.
\] 
\item[(c)] For each $\delta \in S$, the sequence
 $\seq{ \cf(\alpha_{\delta, i+1})}{i < \mu})$ is strictly increasing with 
  limit  $\delta$. By the choice of $S$, $\delta$ is a cardinal. 
  \item[(d)]
Suppose that $E$ is a club of $\kappa$. Then for stationarily many $\delta \in S$  we have:
\begin{itemize}
\item[•] for any large enough $i < \omega$, $\alpha_{\delta, i+1} \in E$, if $\mu = \omega$;
\item[•] $\forall^{\rm club}i < \mu, \alpha_{\delta, i+1} \in E$, if $\mu > \omega$.
\end{itemize}
\item[(e)] The function $f \colon \kappa \to \kappa$ satisfies $f(\beta)< \min(\beta, 1)$ for any $\beta < \kappa$.
\item[(f)] For $\delta \in S':= S \cap C_{{\rm fix}, \kappa}$ and $\beta < \delta$, we let
\begin{equation*}
 S_{\delta, i, \beta} := \{\gamma < \alpha_{\delta, i+1} : f(\gamma) = \beta\}.
 \end{equation*}
We require: For any $\delta \in S'$ and any $\beta < \delta$ the statement
\begin{equation*}\label{code1a}
 S_{\delta, i, \beta} \mbox{ is stationary in }\alpha_{\delta, i+1}.
 \end{equation*}
holds
\begin{itemize}
\item[•] for any large enough $i < \omega$, if $\mu = \omega$;
\item[•] $\forall^{\rm club} i < \mu$, if $\mu > \omega$.
\end{itemize}

\end{enumerate}\end{enumerate}\end{definition}

\begin{lemma}\label{b5}
Let  $\kappa$ be weakly inaccessible and $\mu < \kappa$ be regular and let $S \subseteq \kappa \cap \cof(\mu)$ be stationary. Then
$\boxplus_{\kappa, \mu, S}$ holds. Indeed, for the function $f$ as below for any club $E$ in $\kappa$, for stationarily many $\delta \in S \cap E$   clauses (d), (e) and (f) hold simultanously, i.e, for any $\beta < \delta$,
\begin{equation*}\label{code1b}
S_{\delta, i, \beta} \mbox{ is stationary in }\alpha_{\delta, i+1} \wedge \alpha_{\delta, i+1} \in E
 \end{equation*}
holds
\begin{itemize}
\item[•] for any large enough $i < \omega$, if $\mu = \omega$;
\item[•] $\forall^{\rm club} i < \mu$, if $\mu > \omega$.
\end{itemize}

\end{lemma}

\begin{proof} Given $S$, we first form $S'= S \cap C_{{\rm fix}, \kappa}$. Then we apply \autoref{club_guessing} to $S'$ and get  $\bar{C}= \seq{\seq{\alpha_{\delta, i}} { i < \mu }}{\delta \in S'}$ that fulfils 
$\boxplus_{\kappa, \mu, S}$(a), (b), (c)\footnote{For arranging (c), we possibly thin out $C_\delta$ to a sub-club.} and (d) for some $\delta \in S'$.
Here (d) is true for all $i < \mu$.
 We let $f \colon \kappa \to \kappa$  be defined via
\begin{equation*}\label{the_f1}
f(\gamma) = \begin{cases} \beta, & \mbox{ if } \cf(\gamma) = \aleph_{\beta+1};\\
0, & \mbox{ else.}
\end{cases}
\end{equation*} 
We show that $\boxplus_{\kappa, \mu, S}$(f) holds:   Now $\delta \in S'$ is a fixed point of the $\aleph$-operation. For any $\beta < \delta$, $\aleph_{\beta+1} < \delta$. Since
$\sup_{i<\mu} \cf(\alpha_{\delta, i+1}) = \delta$, there is an end segment of $i < \mu$ such that for each $i $ in this end segment, $\cf(\alpha_{\delta, i+1}) > \aleph_{\beta +1}$. Thus the two bullet points in $\boxplus_{\kappa, \mu, S}$(f) are true: For $i$ in this end segment, $S_{\delta, i, \beta} = \{\gamma < \alpha_{\delta, i+1} : f(\gamma) = \beta\}$ is stationary in $\alpha_{\delta, i+1}$.
\end{proof}

\begin{lemma}\label{f8_kappa_mu} 
Let  $\kappa$ be weakly inaccessible and $\mu < \kappa$ be regular. Let  $S \subseteq \kappa \cap \cof(\mu)$ be stationary. If $\boxplus_{\kappa,\mu, S}$ holds, then
 $\qb \Vdash \diamondsuit_\kappa(S)$.
\end{lemma}

\begin{proof} We let $S' = S \cap \{\alpha< \kappa: \aleph_\alpha = \alpha\}$ and fix  $\bar{C}= \la C_\delta : \delta \in S' \ra$, $\bar{\alpha}$, $f$  as in $\boxplus_{\kappa,\mu,S}$. 
Recall that $\name{\eta}$ is defined in \autoref{name_eta} and $S_{\delta, i, \beta}$. 

 We define the name $\langle \name{\nu}_\delta : \delta \in S \rangle$ by letting for $\delta \in S'$ % \cap C_{\rm fix}$ that is chosen with item (f) of  $\boxplus$, 
 and for any $\beta < \delta$, $j = 0,1$,
\begin{equation} \label{diamondname_lemma4.6}
\begin{split}
\qb \Vdash 
&\mbox{``}\name{\nu}_\delta(\beta) = j \mbox{ iff }\\
&
\forall^{\rm club} i \in \mu \bigl(\exists^{\rm stat} \alpha < \alpha_{\delta, i+1} (f(\alpha) = \beta)\\
& \rightarrow \forall^{\rm club} \alpha < \alpha_{\delta, i+1} 
(f(\alpha) = \beta \rightarrow \name{\eta}(\alpha) = j)\bigr)\mbox{''}.
\end{split}\end{equation}
Again for $\delta \in S \setminus S'$ we can take any name $\name{\nu}_\delta$ for a function in ${}^\delta 2$.

We show 
\begin{equation*}\label{oplus1}
\qb \Vdash  \mbox{``}\seq{\name{\nu}_\delta}{\delta \in S} \mbox{ is a $\diamondsuit_\kappa(S)$-sequence.''} 
\end{equation*}

Towards this suppose that 
\begin{equation*}
p \Vdash \mbox{``}\name{\xname} \in {}^\kappa 2, \mbox{ and } \name{D} \mbox{ is a club subset of }\kappa.\mbox{''}
\end{equation*}
We show that there is some $q \geq p$ that forces $\delta \in \name{D}\cap S'$ and $\name{\xname} \rest \delta = \name{\nu}_\delta$.

We let $\chi = \beth_\omega(\kappa)$ and let $<^*_{\chi}$ be a well-ordering of $H(\chi)$. We choose a $\kappa$-approximating sequence (\autoref{kappaapproxseq}) 
$\seq{N_\eps}{\eps < \kappa}$ with the following properties: 
\begin{equation*}%\tag*{\quad $(\ast)_2$}
{\bf c} = (\kappa, p, S, \bar{C}, \bar{\alpha}, \name{\bar{\nu}}, \name{\xname}, \name{D}) \in N_0.
\end{equation*}
Again we let $E = \{\alpha < \kappa: N_\alpha \cap \kappa = \alpha \}$.
We pick any $\delta$ with 
\begin{equation}\label{delta}%\tag*{\quad $(\ast)_2$}
 \delta \in S'\cap \acc(E) \wedge
\forall^{\rm club} i < \mu (\alpha_{\delta, i+1} \in \acc(E)).
\end{equation}
For $\eps < \delta$, we let $\kappa_\eps = N_\eps \cap \kappa$.

By induction on $\eps \leq \delta$
we chose a condition $p_\eps$ such that for any $\eps$ for any $\zeta < \eps$ the following holds
\begin{enumerate}\item[$(\otimes)$] $p_\eps$ is the $<^*_\chi$-least element such that 
\begin{enumerate} 
\item[(a)] $p_\eps \geq p$

\item [(b)] $p_\eps \geq p_\eta$ for $\zeta<\eps$. 
\item[(c)] $\lg(\tr(p_\eps)) \geq \kappa_\eps$.
\item[(d)] $p_\eps$ forces values to $\name{\xname} \rest \kappa_\eps$, $\name{D} \cap \kappa_\eps$ and $\min (\name{D} \setminus \kappa_\eps)$ call them $\xname_\eps$, $e_\eps$, $\gamma_\eps$ respectively.
\item[(e)] 
for limit $\eps$, $\tr(p_\eps)(\kappa_\eps) = \xname_\eps(f(\kappa_\eps))$. 
\end{enumerate}
\end{enumerate}
We can carry the induction since $\bQ$ is $(<\kappa)$-complete and for clause $(\otimes)$(e) we recall $f(\kappa_\eps)< \kappa_\eps$.

This is literally like the proof of $\circledast_\eps$  in the proof of \autoref{main_theorem_easy_case}.

After carrying the induction, we let $q = \bigcap_{\eps < \delta} p_\eps$.
Since $\delta \in E$, we have $\kappa_\delta = \delta$.

We show that 
\begin{equation}\label{diam} q \Vdash \name{\xname} \rest \delta= \name{\nu}_\delta.
\end{equation}

According to \eqref{diamondname_lemma4.6}, this means for any $\beta < \delta$, for any $j \in 2$,
\begin{equation} \label{diamondname2}
\begin{split}
q \Vdash 
&\mbox{``}\name{\xname}(\beta) = j \mbox{ iff }\\
&
\forall^{\rm club} i \in \mu \bigl(\exists^{\rm stat} \alpha < \alpha_{\delta, i+1} (f(\alpha) = \beta)\\
& \rightarrow \forall^{\rm club} \alpha < \alpha_{\delta, i+1} 
(f(\alpha) = \beta \rightarrow \name{\eta}(\alpha) = j)\bigr)\mbox{''}.
\end{split}\end{equation}

Since $\bar{N}$ is a continuous sequence, the function
$\eps \mapsto \kappa_\eps$ is continuously increasing.
Hence $\kappa_\delta = \delta$ and 
 \begin{equation}\label{pinned}
 q \Vdash \name{\xname} \rest \delta= \bigcup\{\xname_{\kappa_\eps}: \eps < \delta\}.
 \end{equation}

Since $q$ forces $\name{\xname}(\beta) = 0$ or $q$ forces $\name{\xname}(\beta) = 1$, it suffices to check the forward direction in the ``iff'' in Statement \eqref{diamondname2}. So suppose that $q \Vdash \name{\xname}(\beta) = j$. We verify that at club many $i < \mu$  there are stationarily many $\alpha < \alpha_{\delta,i+1}$ with $f(\alpha)  = \beta$, and for club many $\eps < \alpha_{\delta, i+1}$, 
\begin{equation}\label{arranged}
q \Vdash (\kappa_\eps = \eps \wedge f(\eps) = \beta) \rightarrow \eta(\eps) = \name{\xname}(\beta) = \xname_\eps(\beta).
\end{equation}
 Since $\delta \in S'$ and since $\eps \mapsto \kappa_\eps$ is continuous from $\kappa$ to $\kappa$, the premise in \eqref{diamondname_lemma4.6} (i.e., the definition of $\name{\nu}_\delta$), $(\forall^{\club }i < \mu )(\exists^{\rm stat} \alpha < \alpha_{\delta, i+1})( f(\kappa_\alpha) = f(\alpha) = \beta)$ holds. This is not a forcing statement. Namely this statement holds for the $i$ such as in $\boxplus_{\kappa, \mu, S}(e)$.
 
 In each limit $\eps$, we respected $(\otimes)$(e). By the choice of $\delta$ according to Statement \eqref{delta}, for club many $i < \mu$,
$\alpha_{\delta, i+1} \in \acc(E)$ and hence for these $i$,  the set of $\eps < \alpha_{\delta, i+1}$ with $\kappa_\eps = \eps$ is a club in $\alpha_{\delta, i+1}$. So Statement~\eqref{arranged} holds.
Together with Statement \eqref{pinned} and the definition of the name
in Statement~\eqref{diamondname_lemma4.6} this shows that  Statement~\eqref{diam} is true.
\end{proof}

\subsection{Weakening $(<\kappa)$-Closure to a Strong Form of Strategic Closure}\label{S4_2}

We recall, a forcing $\bQ$ is  $\kappa$-strategically complete if the following holds: For any  $p \in \bQ$ there is a there is winning strategy in the game $G(p,\kappa)$ for player COM. The game $G(p, \kappa)$  is as follows.  Player COM starts with $p_0= 1_{\bP}$ and player INC plays in any round $q_\alpha \geq p_\alpha$. In successor rounds  COM plays $p_{\alpha+1} \geq q_\alpha$.
In limit rounds $\delta < \kappa$, Player COM plays $p_{\delta} \geq q_\alpha$ for $\alpha < \delta$.
Player COM wins if $p_\delta$ exists for any $\delta \in \kappa$, otherwise player INC wins.
 
 If $\sigma$ is a winning strategy for COM and and COM modifies this strategy by first picking a move according to $\sigma$ and thereafter strengthening it, then this is a winning strategy as well, since INC could have played this stengthening.  
 
Under $\boxplus_{\kappa, \mu, S}$, we may consider the following property.

\begin{enumerate}
\item[${\rm Pr}(\kappa, \mu, S, \bQ)$:]
$\bQ$ is a $\kappa$-strategically complete forcing and there 
is a name $\name{\tau} = \seq{\name{\tau}_\eps}{\eps < \kappa}$, such that for any $p \in \bQ$ there is a winning strategy for COM in 
$G(p,\kappa)$ 
 with the following property: 
 In any play played according to this strategy:
For a club $C$ in $\kappa$ for $\eps \in C$ for $j = 0,1$, there are upper bounds 
$r_{\eps,0}$, $r_{\eps,1}$ of $\seq{p_\zeta, q_\zeta}{\zeta < \eps}$ with $r_{\eps,j} \Vdash \name{\tau}_\eps = j$.
\end{enumerate}

\begin{theorem}
Suppose $\boxplus_{\kappa, \mu, S}$  and that $\bQ$ is a $\kappa$-strategically closed forcing with ${\rm Pr}(\kappa, \mu, S, \bQ)$.
Then  $\bQ \Vdash \diamondsuit_\kappa(S)$.
\end{theorem}
\begin{proof} (Sketch)
We modify the original proof by adding that the strategy is an element of $N_0$, the first element of an $\kappa$-approximating sequence. 
We work with the following diamond $\seq{\name{\nu}_\delta}{\delta \in S}$ such that for $\delta \in S$ for any $q \in \bQ$:
\begin{equation*} \label{diamondname6}
\begin{split}
q \Vdash 
&\mbox{``}\name{\nu}_\delta(\beta) = j \mbox{ iff }\\
&
\forall^{\rm club} i \in \mu \bigl(\exists^{\rm stat} \alpha < \alpha_{\delta, i+1} (f(\alpha) = \beta)\\
& \rightarrow \forall^{\rm club} \alpha < \alpha_{\delta, i+1} 
(f(\alpha) = \beta \rightarrow \name{\tau}_\alpha = j)\bigr)\mbox{''}.
\end{split}
\end{equation*}

We proceed as in the proof of \autoref{b14}. There, clause $(\otimes)$(e) says $p_\eps \Vdash \tr(p_\eps)(\kappa_\eps) = x_\eps(f_\eps(\kappa_\eps))$. On the club set of $\eps$ with $\kappa_\eps = \eps$, now  player COM chooses $j\in 2$ such that that
  $p_\eps = r_ {\eps, j}\Vdash \tau_\eps = x_\eps(f_\eps(\eps))$.
 
\end{proof}

\begin{remark}
We wrote a ``a strong form of strategic closure'', since  ``$(<\kappa)$-strategically closed'' means often that for each $\alpha < \kappa$,  player COM has a winning strategy in $G(p,\alpha)$. A typical example is the forcing adding a $\square_{\aleph_1}$-sequence with $(<\aleph_2)$-sized closed initial segments for $\kappa = \aleph_2$. $(< \kappa)$-strategical completeness does not allow to carry out the above proof, since separate strategies for each $\delta < \kappa$  cannot be all contained as elements in an elementary submodel of size $<\kappa$.   
\end{remark}

\section{Higher Miller Forcing with Splitting into a Club}\label{S5}

Higher Miller forcing and higher Laver forcing are special among our  forcings, since there is a  name for a $\diamondsuit$-sequence in the respective forcing extensions that is much simpler than the other names for diamonds.  Now we prove \autoref{main_theorem_3}. The proof works for either of these two forcings.

\begin{proof} (1) Let $\kappa$ be a regular uncountable cardinal. We give a $\qa$-name that witnesses $\qa \Vdash \diamondsuit_\kappa$.

Let $\seq{S_\alpha}{\alpha \in \kappa}$ be a partition of $\kappa$ into stationary sets. 
For each $\alpha < \kappa$, 
we let $\seq{t_{\alpha, \eps}}{\eps < \kappa}$ be 
an enumeration of ${}^\alpha \kappa$.
For $i<\kappa$ we let $u_{\alpha,i} = t_{\alpha, \eps}$ if $i \in S_\eps$. 
Recall, $\name{\eta}$ is a name of the generic branch.
Now we give a name for a $\diamondsuit_\kappa(S)$-sequence:
\[
\qa \Vdash \name{\bar{d}} = \seq{\name{d}_\alpha= u_{\alpha, \name{\eta}(\alpha)}}{\alpha \in \kappa}.
\]
We show 
\[\qa \Vdash \mbox{``}\name{\bar{d}} \mbox{ is a $\diamondsuit_\kappa$-sequence.''}
\]

We assume $p \Vdash \name{x} \in {}^\kappa \kappa, \name{C} $ is a club in $\kappa$. We show that there are some $\alpha < \kappa$ and a stronger condition $q$ that forces $\alpha \in \name{C}$ and $\name{x} \rest \alpha = \name{d}_\alpha$.
By induction on $n < \omega$ we choose $p_n$ and $\alpha_n \in \kappa$ such that
\begin{enumerate}
\item[(a)] $p_0=p$,
\item[(b)] $p_n \leq p_{n+1}$,
\item[(c)] $\alpha_n < \dom(\tr(p_n)) \leq \alpha_{n+1}$,
\item[(d)] $p_n \Vdash \alpha_n \in \name{C}$,
\item[(e)] $p_{n+1} $ forces a value in $\bV$ to $\name{\xname} \rest \alpha_n$, we call it $\xname_n$.
\end{enumerate}
The induction can be carried since the forcing is $(<\kappa)$-closed and hence does not add new elements to ${}^{\kappa>} \kappa$.
Also by closure, the set $p_\omega = \bigcap\{p_n : n < \omega\}$ is a condition. We let $\alpha = \sup_n \alpha_n$. Then $\dom(\tr(p_\omega)) = \alpha$.  
We let $\bigcup\{\xname_n : n < \omega\} = \xname_\omega$ and notice $\xname_\omega \in {}^\alpha \kappa$.
By construction,
\[p_\omega \Vdash \name{x} \rest \alpha = v_\omega \wedge \alpha \in \name{C}.
\]
Now we strengthen $p_\omega$  by a trunk lengthening:  
The set $\osucc_{p_\omega}(\tr(p_\omega))$ is a club subset of $\kappa$ and thus has non-empty intersection with each $S_\eps$, $\eps < \kappa$. We choose $\eps$ to be an $\eps$ with $t_{\alpha, \eps} = \xname_\omega$.  We pick some 
$i \in S_\eps \cap \osucc_{p_\omega}(\tr(p_\omega))$. Then $u_{\alpha, i } = t_{\alpha,\eps}$.
Now \[
p_\omega^{\langle \tr(p_\omega) \concat \langle i \rangle\rangle} \Vdash \name{\eta}(\alpha) = i \wedge \name{d}_\alpha = u_{\alpha,i} = t_{\alpha, \eps} = \xname_\omega = \name{x} \rest \alpha.
\]
\medskip

(2) Let a stationary set $S \subseteq \kappa$ be given.
We work with the same $\seq{S_\eps}{\eps<\kappa}$,  $\seq{t_{\alpha, \eps}}{\eps < \kappa}$ and $u_{\alpha, i}$ for $i \in S_\eps$ and $\alpha < \kappa$ as above. We give a name for a $\diamondsuit_\kappa(S)$-sequence:
\[
\name{\bar{d}} = \seq{\name{d}_\alpha= u_{\alpha, \name{\eta}(\alpha)}}{\alpha \in S}.
\]
We show 
\[\qa \Vdash \mbox{``}\name{\bar{d}} \mbox{ is a $\diamondsuit_\kappa(S)$-sequence.''}
\]
Let $p \Vdash \mbox{``}\name{C}$ is club in $\kappa$ and
$\name{\xname} \subseteq \kappa$''.
We show that there is $q \geq p$ and $\delta \in S$ with 
\begin{equation}\label{aim}
q \Vdash \delta \in \name{C} \wedge \name{\xname} \rest \delta = \name{d}_\delta.
\end{equation}
 
We let $\chi = \beth_\omega(\kappa)$ and let $<^*_{\chi}$ be a well-ordering of $H(\chi)$. We choose a $\kappa$-approximating sequence (\autoref{kappaapproxseq}) 
with \begin{equation*}
 {\bf c} = (\kappa, p, \name{\bar{d}}, \name{\xname}, \name{D}, S) \in N_0.
\end{equation*}
Again we let $E = \{\alpha < \kappa: N_\alpha \cap \kappa = \alpha \}$.
We pick any $\delta$ with 
\begin{equation*}%\tag*{\quad $(\ast)_2$}
 \delta \in S \cap \acc(E) 
\end{equation*}
For $\eps < \delta$, we let $\kappa_\eps = N_\eps \cap \kappa$.

By induction on $\eps \leq \delta$
we chose a condition $p_\eps$ such that for any $\eps$ for any $\zeta < \eps$ the following holds
\begin{enumerate}\item[$(\oplus)$] $p_\eps$ is the $<^*_\chi$-least element such that 
\begin{enumerate} 
\item[(a)] $p_\eps \geq p$

\item [(b)] $p_\eps \geq p_\eta$ for $\zeta<\eps$. 
\item[(c)] $\lg(\tr(p_\eps)) \geq \kappa_\eps$.
\item[(d)] $p_\eps$ forces values to $\name{\xname} \rest \kappa_\eps$, $\name{D} \cap \kappa_\eps$ and $\min (\name{D} \setminus \kappa_\eps)$ call them $\xname_\eps$, $e_\eps$, $\gamma_\eps$ respectively.
\item[(e)] 
for limit $\eps$, $\tr(p_\eps)(\kappa_\eps) = \xname_\eps(f(\kappa_\eps))$. 
\end{enumerate}
\end{enumerate}

As in the proof of \autoref{f8} or of \autoref{b14}, it is shown that we can carry the induction.

Now we carried the induction. We let $q = \bigcap_{\eps < \delta} p_\eps$.

Since $\delta \in E$, we have $\kappa_\delta = \delta$.

Now we strengthen $p_{\delta}$  by a trunk lengthening:  
The set $\osucc_{p_{\delta}}(\tr(p_{\delta}))$ is a club subset of $\kappa$ and thus has non-empty intersection with each $S_\eps$, $\eps < \kappa$. There is some $\eps< \kappa$ with $t_{\alpha, \eps} = \xname_{\delta}$.  We pick some 
$i \in S_\eps \cap \osucc_{p_{\delta}}(\tr(p_{\delta}))$. Then $u_{\alpha, i } = t_{\alpha,\eps}$.

Putting all together, we get  
\[
p_{\delta}^{\langle \tr(p_\delta) \concat \langle i \rangle\rangle} \Vdash \name{\eta}(\alpha) = i \wedge \name{d}_\alpha = u_{\alpha,i} = t_{\alpha, \eps} = \xname_{\delta} = \name{\xname} \rest \delta,
\]
and $p_{\delta}^{\langle \tr(p_\delta) \concat \langle i \rangle\rangle} = q$ witnesses Statement~\eqref{aim}.

Part (3) is proved in the proof of \autoref{cor2}.
\end{proof}

This concludes the proof of \autoref{main_theorem_3}.

\begin{remark}  The forcing adds a $\kappa$-Cohen real $\mathbb C_\kappa$. This is shown in \cite{BrendleBrooke-TaylorFriedmanMontoya}. So there is 
a $\mathbb C_\kappa$-name of a diamond sequence in $V[{\mathbb C}_\kappa]$.
\end{remark}

\section{On $\qbW$ and Successor Cardinals $\kappa$}\label{S6}

Now we work with $\kappa^{<\kappa} = \kappa \geq \aleph_1$ and allow $\kappa$ to be a successor cardinal. We present another type of name of a diamond. The technique works for a weakening of the demands on splitting nodes in forcing conditions. The diamonds guess at approachable ordinals. Therefore the approachability ideal becomes important. We show that under $\kappa^{< \kappa} > \kappa$ the forcing adds a collapse from $\kappa^{<\kappa}$ to $\kappa$. Our collapsing technique is different from \cite[Section 4]{MdSh:1191}.
The results in this section pertain to any of the mentioned tree forcings and their $W$-variants as given by the pattern \autoref{l2}. For simplicity, we focus on $\qbW$.

\begin{definition}\label{l2}
Let $\kappa = \cf(\kappa) > \omega$ and let $W \subseteq \kappa$ be stationary in $\kappa$. We let $\bQ = \qbW$ be the forcing notion that is defined as follows
\begin{enumerate}
\item[(A)] $p \in \qbW$ if 
\begin{enumerate}
\item[(a)] $p$ is a non-empty subtree of ${}^{\kappa>} 2$. 
\item[(b)] $p$ is closed under unions of $\triangleleft$ increasing sequences of lengths $(< \kappa)$. We say $p$ is $(<\kappa)$-closed.
\item[(c)] The set of spitting nodes is dense. That is for any $\eta \in p$ there is $\nu \in p$ with $\eta \trl \nu$ and $\nu \concat \la 0 \ra, \nu \concat \la 1 \ra \in p$.
\item[(d)] If $\delta$ is a limit ordinal and $\seq{ \eta_\eps}{\eps < \delta}$ is a $\triangleleft$-increasing sequence with $\bigwedge_{\eps<\delta} \eta_\eps \in \Split(p)$ and $\bigcup_{\eps<\delta} \lg (\eta_\eps) \in W$, then $\bigcup_{\eps < \delta} \eta_\eps \in \Split(p)$.
\end{enumerate}
\item[(B)] $p \leq q$ if $p \supseteq q$.
\item[(C)] The generic is $\name{\eta} = \bigcup\{\tr(p) : p \in \name{\bG}_{\bQ}\}$.
\end{enumerate}
\end{definition}

\begin{fact}\label{l3} Let $W \subseteq \kappa$ be stationary.
 \begin{enumerate}
\item[1)]
$\qbW$ is closed under intersections of increasing chains of length $< \kappa$.
\item[2)] $\bQ \Vdash \name{\eta} \in {}^\kappa 2$ and $\bV[\name{\eta}] = \bV[\name{\bG}_{\bQ}]$. 
\end{enumerate}
\end{fact}

\begin{proof} 1) Let $\seq{p_\alpha}{\alpha< \gamma}$ be an increasing sequence of conditions and $\gamma < \kappa$. The intersection $q = \bigcap_{\alpha< \gamma} p_\alpha$ has properties \autoref{l2}(a), (b) and (d). 
Since $\emptyset \in q$, it is non-empty. We have to show that $q$ is a perfect tree. We first show
\begin{enumerate}
\item[($\ast$)]
For any $s \in q$,  $q$ has a branch $b$ containing $s$ and containing cofinally many splitting nodes in $p_\alpha$ for each $\alpha < \gamma$.
\end{enumerate}

Since $q$ is $(< \kappa)$-closed, it suffice to show:
\begin{enumerate}
\item[($\ast'$)]
For any $s \in q$ there is a strict extension $t \triangleright s$, $t \in q$ and $t \in \Split(q)$ or for any $\alpha < \gamma$, $t $ is a limit of splitting nodes in $p_\alpha$.
\end{enumerate}

Let $s \in q$ be given. We choose $i \in 2$ such that $s' = s \concat \la i \ra \in q$.
First case: $s' \in \Split(p_\alpha)$ for any $\alpha < \gamma$. Then $s' \in \Split(q)$ and we are done.
Second case: There is $j \in 2$, $\alpha_0$ such that $ s' \concat \la j \ra \trl \tr(p_{\alpha_0})$. We let $s_0 = \tr(p_{\alpha_0})$.

Given  $\seq{s_j= \tr(p_{\alpha_i})}{j < i\leq \gamma}$ that is $\trl$-increasing and $\seq{\alpha_j}{j < i}$ that is increasing, in the limit case, we let $s_i$ be the union of the $s_j$ and $\alpha_i$ be the supremum of the $\alpha_j$. In the successor case $i = j+1< \gamma$, if $s_j \in \Split(p_\alpha)$, $\alpha < \gamma$, then ($\ast'$) is proved with $t = s_j$. If not, we let $s_i = \tr(p_{\alpha_i}) \triangleright s_j$ for the minimal $\alpha_i > \alpha_j$ such $\tr(p_{\alpha_i}) \triangleright s_j$.
If for any $i < \gamma$, we are always in the second case, then $t=s_\gamma$ is as in $(\ast')$.

So $(\ast')$ and $(\ast)$ are proved and hence for any $t \in q$ there is some $b \in [q]$ with $t \in b$ such that for any $\alpha < \gamma$ there are cofinally many splitting nodes of $p_\alpha$ on this branch $b$.

We show that $q$ fulfils \autoref{l2}(c). Let $t \in q$.  We have to show that there is a splitting node above $t$. We take a branch $b$ containing  $t$ as above.
Now by \autoref{l2}(c), for each $\alpha < \gamma$, the  set
$W_{b,\alpha}= \{\beta \in \kappa : (\exists t' \in b \cap \Split(p_\alpha))(\dom(t') = \beta)\}$ is a  superset of a club $C_{\alpha}$ in $\kappa$ intersected with $W$. We can let $C_\alpha = \acc^+(W_{b,\alpha})$.  Now the intersection of the $W_{b,\alpha}$, $\alpha < \gamma$, is a superset of $\bigcap\{C_\alpha : \alpha < \gamma\} \cap W$.
Hence there is a splitting node $t' \in \Split(q)$ on the branch $b$ with $t' \trianglerighteq t$, namely any $t' \in b$ with $\dom(t') \in \bigcap\{C_\alpha : \alpha < \gamma\} \cap W \cap [\dom(t), \kappa)$ has these properties.

2)  Let $\bG$ be $\bQ$-generic over $\bV$. Then $\eta= \bigcup\{\tr(p) : p \in \bG\}$ is a function from $\kappa$ to $2$, since for any $p \in \bQ$ and any $t \in p$ also the subtree $p^{\langle t \rangle}$ is a condition, and if $t$ and $t'$ are incompatible nodes in $p$, the conditions $p^{\langle t \rangle}$ and $p^{\langle t'\rangle}$ are incompatible.  The generic branch $\eta$ contains the full information about ${\bG}$ since for any generic filter $\bG$ we have for any $p$: $p \in \bG$ iff $\eta \in [p]$.
For a detailed proof see \cite[Proposition 1.2]{MdSh:1191}.
\end{proof}

\begin{fact}\label{l4} Assume $\kappa > \aleph_0$ is regular and $W \subseteq \kappa$ is stationary and $\bQ = \qbW$. 
If $\kappa= \kappa^{<\kappa}$, then $\bQ$ is proper.
\end{fact}

 \begin{proof}  This is proved by a routine fusion construction. The proof that $\bQ$ is $(<\kappa)$-closed and has $\kappa$-long fusion sequences with limits and is similar to the proof of \autoref{claim2.21}. \end{proof}

We introduce some combinatorics that will be useful for defining names.

\begin{definition}\label{treeembedding}
Suppose that $\delta \in \kappa$ and $\cf(\delta) = \cf(\sigma)$.
\begin{enumerate}
\item[(A)]  A function
$f \colon {}^{\sigma >} 2 \to {}^{\delta>} 2$ is called 
a $\sigma$-tree embedding of height $\delta$, if the following holds:
\begin{enumerate}
\item[(a)] for any $s, t \in {}^{\sigma >} 2$, if $s \trl t$, then $f(s) \trl f(t)$. 
\item[(b)] For any $b \in {}^\sigma 2$, $\bigcup\{f(b \rest i) : i < \sigma\} \in {}^\delta 2$.  
\end{enumerate}
We let $\sigma = \lim \seq{\sigma_i}{i < \cf(\sigma)}$ for an increasing continuous sequence. It suffices to know 
$f\rest  \bigcup\{{}^{\sigma_i} 2 : i < \cf(\sigma)\}$ and define the other $f$ values by  a natural modification of (b). 
\item[(B)]
A $\sigma$-tree embedding of height $\delta$ is called one-to-one if in addition
for any $s, t \in {}^{\sigma >} 2$, if $s \perp t$ then $f(s) \perp f(t)$. We write $s \perp t$, if $s$ and $t$ are incomparable (which is the same as incompatible) in $\trl$.
\item[(C)]
Let $f$ be a $\sigma$-tree embedding of height $\delta$. Let $\seq{\sigma_i }{i < \cf(\sigma)}$ be an increasing cofinal sequence in $\sigma$. A sequence $\seq{(\sigma_i,\ell_i)}{i< \cf(\sigma)}$ is a height sequence for $f$, if for any
$i < \cf(\sigma)$, for any $t \in {}^{\sigma_i} 2$, $\dom(f(t)) \in [\ell_i, \ell_{i+1})$. Necessarily $\lim_{i<\cf(\sigma)} \ell_i = \delta$. 
\item[(D)] Given a $\sigma$-tree embedding $f$ of height $\delta$, there is a lift to branches $\bar{f} \colon {}^\sigma 2 \to {}^\delta 2$ given by
$\bar{f}(b) = \bigcup\{f(b\rest \sigma_i) : i < \cf(\sigma)\}$.  
\end{enumerate}
 \end{definition}
 
\begin{remark}\hfill
\begin{enumerate}
\item[(a)]
Height sequences  do not need to exist. If $2^{<\sigma} < \kappa$ and $\kappa$ is regular, then for $i < \cf(\sigma)$, the heights in $\{f(t): t \in 2^{\sigma_i}\}$ are bounded. In our inductive construction of a $\sigma$-tree of conditions in the proof of \autoref{l5}, we will naturally define two tree embeddings with the same height sequence. These are the $(f_1,f_2)$ in the end of the proof of \autoref{l5}. In the proof of \autoref{l4(2)} there are just fronts of heights.

\item[(b)]
We do not require that the tree emdeddings fulfil $f(s \cap t) = f(s) \cap f(t)$. The righthand side might be longer. \autoref{l2}(A) entails that for limits of splitting nodes in a condition $p$ with domain not in $W$ we cannot expect prompt splitting. 
\end{enumerate}\end{remark}

The following lemma is used for names of diamonds and for names of collapsing functions.

\begin{lemma}[Bernstein Lemma]\label{Bernstein_lemma}
We assume that $\kappa = \kappa^{<\kappa}$ and $2^\sigma = \kappa$ and $\kappa^{2^{<\sigma}} = \kappa$. For each $\delta \in  \kappa \cap \cof(\cf(\sigma))$ we let 
\begin{equation*}\label{Fdelta}
\begin{split}
 \cF_{\sigma,\delta} = & 
\{(f_1,f_2) : f_1, f_2 \mbox{ are $\sigma$-tree embeddings}\\
&\mbox{ of height $\delta$ and $f_1$ is one-to-one}\}.
\end{split}
\end{equation*}

Then there is some $h_\delta \colon {}^\delta 2 \to {}^\delta 2$ such that
\begin{equation*}\begin{split}
&\forall (f_1, f_2) \in \cF_{\sigma,\delta}) \Bigl((\exists \eta \in {}^\sigma 2)\\
&(h_\delta(\bar{f_1}(\eta)) = \bar{f_2}(\eta) \wedge (\forall \alpha \in {}^\delta 2) (\exists \eta' \in {}^\sigma 2) h_\delta(\bar{f_1}(\eta')) = \alpha)\bigr)\Bigr).
\end{split}
\end{equation*}
\end{lemma} 

\begin{proof}
For $\delta \in \kappa \cap \cof(\cf(\sigma))$ we have $|\cF_{\sigma,\delta}| \leq \kappa$.
Note that here we use $2^{<\sigma} <\kappa$, $2^{<\delta} \leq \kappa$  and $\kappa^{2^{<\sigma}} = \kappa^{<\kappa} = \kappa$.

 We enumerate 
\[\{ (f_1,f_2,x) : (f_1, f_2) \in \cF_{\sigma,\delta},
x \in {}^\delta 2\}
\]
as  
$\seq{(f_1^\alpha, f_2^\alpha,x_\alpha)}{\alpha < \kappa}$ such that each triple appears $\kappa$ often.
At step $\alpha$ we have to take care of $(f_1^\alpha, f_2^\alpha)$ and we have to ensure that $x_\alpha$ gets into the range of $h_\delta\circ \bar{f_1^\alpha}$.

We define $\eta_\alpha, z_\alpha$ and $h_\delta (\bar{f_1^\alpha}(\eta_\alpha)):=\bar{f_2^\alpha}(\eta_\alpha)$
and $h_\delta(\bar{f_1^\alpha}(z_\alpha)) := x_\alpha$ by induction on $\alpha$.
Suppose that $\seq{(\eta_\beta, z_\beta, h_\delta(f_1^\beta(\eta_\beta)), h_\delta(\bar{f^\alpha_1}(z_\beta))}{\beta < \alpha}$ is defined.

Since $\bar{f_1^\alpha}$ is one-to-one, there is some $\eta=\eta_\alpha \in {}^\sigma 2 \setminus \{\eta_\beta : \beta < \alpha\}$ such that $\bar{f_1^\alpha}(\eta) \neq
\bar{f_1^\beta}(\eta_\beta)$ and for each $\beta < \alpha$. 
We let $h_\delta(\bar{f^\alpha_1}(\eta_\alpha)) = \bar{f_2^\alpha}(\eta_\alpha)$ and
we can pick some $z_\alpha \in {}^\sigma 2 \setminus\{\eta_\beta: \beta \leq \alpha\}$  and let $h_\delta(\bar{f_1^\alpha}(z_\alpha)) = x_\alpha$. Here we use that $f_1^\alpha$ is one-to-one. If after the induction the domain of $h_\delta$ is not yet the full set ${}^\delta 2$, we can define $h_\delta$ at the remaining arguments in an arbitrary manner.
\end{proof}

\begin{remark}\label{l6}
 \autoref{l5} and the results in previous sections are incomparable:  
E.g., for $\kappa = \lambda^+$ with $\cf(\lambda) = \aleph_0$ and $S \subseteq \kappa \cap \cof(\aleph_0)$, $S \in \checki[\kappa]$, we can apply \autoref{l5} but not \autoref{b14},
whereas if $S \not\in \checki[\kappa]$ we can apply \autoref{b14} but not \autoref{l5}. We work at $\delta \in S$ with an induction of length $\cf(\delta)$ in the case of approachability,  or of length $\delta$ in general.
\end{remark}

{\bf Proof of \autoref{l5}}\\[-20pt]
\begin{proof} Recall $S \in \checki[\kappa]$, $S \subseteq \kappa \cap \cof(\cf(\sigma))$, $2^{<\sigma} < \kappa$, $2^\sigma = \kappa= \kappa^{<\kappa}\geq \aleph_1$.
For $\delta \in S$ we let $h_\delta$ be as in the \autoref{Bernstein_lemma}.   We define a $\bQ$-name $\name{\bar{\nu}} = \seq{\name{\nu_\delta}}{\delta \in S}$ by 
\begin{equation*}\label{diam4}
\bQ \Vdash \name{\nu}_\delta = h_\delta(\name{\eta}\rest \delta).
\end{equation*}
 We show that  $\bQ$ forces that
 $\name{\bar{\nu}}$  is a $\diamondsuit(S)$-sequence.
 Let 
 \begin{equation*}\label{odot6.1}
 p \Vdash \name{x} \in {}^\kappa 2 \wedge \name{D}\mbox{ is a club in }\kappa. 
 \end{equation*}

We have to find a $\delta \in S$ and some $q \geq p$ such that
\begin{equation}\label{odot6.2}
q \Vdash \delta \in \name{D} \wedge \name{\nu}_\delta = \name{x} \rest \delta.
\end{equation}

 Suppose that $\kappa^{<\kappa} = \kappa > \aleph_0$, $2^\sigma = \kappa$ and $2^{<\sigma} < \kappa$. 

Let for $\delta \in S$, $\sigma$ be as in the assumptions. We let $\seq{\sigma_i}{i<\cf(\sigma)}$ be a cofinal sequence in $\sigma$.

As $S \in \checki[\kappa]$ and $S \subseteq \kappa \cap \cof(\tau)$ there is $(E, \bar{A})$ such that
\begin{enumerate}
\item [$(\odot)_0$] \begin{enumerate}
\item[(a)] $E$ is a club in $\kappa$
 \item[(b)] $\bar{A} = \seq{A_\alpha}{\alpha<\kappa}$,  $A_\alpha \subseteq \alpha$, $A_\alpha$ consists only of successor ordinals,
\item[(c)] For any $\beta \in A_\alpha$ we have $A_\beta = A_\alpha \cap \beta$.
\item[(d)] if $\alpha \in E \cap S$ then $\alpha = \sup(A_\alpha)$ and $\ot (a_\alpha) = \cf(\sigma)$.
\item[(e)] if $\alpha \in \kappa \setminus (E \cap S)$ then $\ot(A_\alpha) < \cf(\sigma)$.
\end{enumerate}
\end{enumerate}

The existence of $\bar{A}$ is derived in the proof of \autoref{club_nacc_charact.} as given in  \cite[Theorem 3.7]{Eisworth_handbook}.

We fix a sequence $\seq{\sigma_i}{i<\cf(\sigma)}$ that is continuously increasing and cofinal in $\sigma$.

We chose by induction on $\alpha < \kappa$ a sequence $\seq{N_\alpha}{\alpha<\kappa}$ such that
\begin{enumerate}
\item [$(\odot)_1$] 
\begin{enumerate}
\item[(a)] $N_\alpha \prec (H(\chi), \in, <^*_\chi)$,
 \item[(b)] $|N_\alpha| < \kappa$,
 \item[(c)] $N_\alpha \cap \kappa \in \kappa$,
 \item[(d)] $N_\alpha$ is $\prec$-increasing and continuous, and $\seq{N_\beta}{\beta \leq \alpha} \in N_{\alpha+1}$,
\item[(e)] ${\bf c} = (\kappa, p, \bar{a}, \seq{A_\delta}{\alpha \in \kappa}, \name{x}, \name{D}, E, S, \seq{\sigma_i}{i < \cf(\sigma)}) \in N_0$.
\item[(f)] For any $\beta < \kappa$, if $\beta \in N_\alpha$, then for any $j < \cf(\sigma)$, ${}^{\sigma_j} \beta \subseteq N_{\alpha+1}$.
Since $2^{<\sigma} < \kappa$, we can add the clause, which will be used to derive $(\odot)_2$ below. 
 \end{enumerate}
 \end{enumerate}

We can assume that $\tau \subseteq N_0$ and hence for any $i < \tau$, $\sigma_i \in N_0$.
Let 
\[C = \{\delta \in E : N_\delta \cap \kappa = \delta\}.
\]
So $C$ is a club of $\kappa$.

For $\delta \in C \cap S$ let $\seq{\gamma_{\delta, i}}{i<\cf(\sigma)}$ list the closure of $A_\delta$ from $(\odot)_0$ in increasing order and let 
\[N_{\delta, i} = N_{\gamma_{\delta, i}}.
\]
Clearly 
\begin{equation}\label{approx_at_delta}
\seq{N_{\delta, \eps}}{\eps \leq i} \in N_{\delta,i+1} \prec N_{\delta,j}
\end{equation} for any $i < j < \sigma$ thanks to the sequence $\bar{A}$ from $(\odot)_0$ and the requirements $(\odot)_1$ for $\alpha = \gamma_{\delta, i}$.

By induction on $i \leq \cf(\sigma)$
we chose a condition $(\bar{p}_i, \bar{x}_i, \bar{\gamma'}_i)$ where $\bar{p}_i = \seq{p_{\sigma_i, \varrho}}{\varrho \in {}^{\sigma_i} 2}$,
such that for any $i<\cf(\sigma)$  the following holds
\begin{enumerate}\item[$(\odot)_{2}$] 
\qquad $p_i$ is the $<^*_\chi$-least element such that 
\begin{enumerate} 

\item[(a)] $\bar{p}_i = \seq{p_{\sigma_i, \varrho}}{\varrho \in {}^{\sigma_i} 2}$. 

(We shall prove that $\bar{p}_i \in N_{\delta,i+1}$ in our construction but it is better  to not state it in our demands.)

\item [(b)] For each $\varrho \in {}^{\sigma_i} 2$, for any $j < i$, $p_{\sigma_j, \varrho \rest \sigma_j} \leq p_{\sigma_i, \varrho}$.

\item[(c)]
 $\bar{x}_i = \seq{x_{\sigma_i, \varrho}}{\varrho \in {}^{\sigma_i} 2}$.

 \item[(d)] For $i < \cf(\sigma)$, $\varrho \in {}^{\sigma_i} 2$, $p_{\sigma_i, \varrho} \Vdash \name{x} \rest \gamma'_{\sigma_i, \varrho} = x_{\eps, \varrho}$
 for some $x_{\sigma_i, \varrho} \in {}^{\gamma_{\sigma_i, \varrho}}2 \cap \bV$.

\item[(e)] for any $\varrho \in {}^{\sigma_i} 2$, $p_{\sigma_i, \varrho} \Vdash\ \gamma'_{\sigma_i, \varrho} \in \name{D}$.

\item[(f)] for any $\varrho \in {}^{\sigma_i} 2$, $p_{\sigma_i, \varrho} \geq p$.

\item[(g)] for any $\varrho \in {}^{\sigma_i} 2$, $\lg(\tr(p_{\sigma_i, \varrho})) \geq \gamma'_{\sigma_i,\varrho}$.

\item[(h)] $\seq{\tr(p_{\sigma_i, \varrho})}{\varrho \in {}^{\sigma_i} 2}$ consists of pairwise $\trl$-incomparable  elements of $p \cap 2^{<\delta}$. 
\item[(i)] for any $\varrho \in {}^{\sigma_i} 2\cap N_{\delta,i+1}$,
 \\$\lg(\tr(p_{\sigma_i, \varrho}))$, $\gamma'_{\sigma_i, \varrho} \in[\sup(N_{\delta,i} \cap \kappa), \sup(N_{\delta,i+1} \cap \kappa))$.
\nothing{ that there are stricly fewer than $\kappa$ many $\varrho$, so the whole tuple is in $N_{i+1}$ by $(\odot)_{3,\sigma_i}$. However, still we do not have $\{p_{\sigma_i, \varrho} : \varrho \in {}^{\sigma_i} 2\} \subseteq N_{i+1}$, since $N_{i+1}$ may be too small. Why do we know $2^{\sigma_i} \subseteq N_{i+1}$?
 I think that the application of the Bernstein Lemma \autoref{Bernstein_lemma} after equation \eqref{f2} 
is a flaw.
\label{stickingout}
}
 \end{enumerate}
\end{enumerate}
We can carry the induction? We can carry the induction since $\bQ$ is $(<\kappa)$-complete.
  More fully, let us prove $(\odot)_{3,i}$ by induction on $i$, 

\begin{enumerate}
\item[$(\odot)_{3,i}$]
\begin{itemize}
\item[•] ${\bf m}_i = \seq{\bar{p}_j, \bar{x}_j, \bar{\gamma'}_j}{j<i}$ exists and is unique and $j<i$ implies ${\bf m}_{j+1} \in N_{\delta,j+1}$ 
and ${}^{\sigma_j} 2$ and the ranges of each of $\bar{p}_j$, $\bar{x}_j$, $\bar{\gamma'}_j$, $j \leq i$ are subsets of $N_{\delta,j+1}$.

\item[•] ${\bf m}_i$ is defined in $(H(\chi), \in , <_\chi^*)$ by a formula $\varphi = \varphi(x, \bar{y})$ with $x $ for ${\bf m}_i$ and $\bar{y} = (y_0, y_1)$ with $y_0 = \seq{N_{\delta, j}}{j < i}$ and $y_1 = {\bf c}$ from $(\odot)_1$(e). 
\end{itemize}

\end{enumerate}

{\bf Case 1} $i = 0$. 

${\bf m}_0 = \langle \rangle$.

{\bf Case 2} $i = j+1$.

  Now $\bar{p}_i$ is the $<_\chi^*$ first ${}^{\sigma_i}2$-tuple of $\bQ$ satisfying clauses $(\odot)_{2}$. As each element of $H(\chi)$ mentioned above is computable from $\seq{N_{\delta, j}}{j < i}$, it belongs as an element to $N_{\delta, i} = N_{\delta,j+1}$  by Equation~\eqref{approx_at_delta}. We  use the specific choice of $\bQ = \qbW$ and not just $(<\kappa)$-complete forcings that force $\name{\eta} \not\in\bV$ but that above each node there is a higher splitting node. This entails that each entry of $\bar{p}_j$ gets ${}^{[\sigma_j,\sigma_i)} 2$-many successors with pairwise incompatible trunks.
%By definability,  $\bar{p}_i \in N_{\delta, i+1}$.
By $(\ast)_\alpha$(f) for $\alpha = \gamma_{\delta, i}$ we have $ N_{\gamma_{\delta,i} +1} \subseteq N_{\gamma_{\delta, i +1}}$. Therefore, ${}^{\sigma_i} 2$ and the ranges of $\bar{p}_i$, $\bar{x}_i$, $\bar{\gamma'}_i$ are subsets of $N_{\delta, i+1}$.

{\bf Case 3} $i$ limit.\\
We again use the specific choice of $\bQ = \qbW$ and not just $(<\kappa)$-complete forcings that force $\name{\eta} \not\in\bV$ but that above the limit of splitting nodes there is a higher splitting node.
(This jump does not harm the continuity of $f_1$ from \eqref{f1} but just makes it not necessarily $\wedge$-preserving. Since $W$ is not necessarily closed, such jumps can appear.)
By $(<\kappa)$-completeness of $\qbW$, we have $p_{i, \varrho} = \bigcap_{\zeta<i} p_{\sigma_\zeta, \varrho \rest \zeta}$ in $\qbW$. Also for $\zeta < i$ $\varrho \in N_{\zeta+1}$, $\lg(\tr(p_{\sigma_\zeta, \varrho \rest \zeta})) \in [\sup(N_{\delta,\zeta} \cap \kappa), \sup(N_{\delta,\zeta+1} \cap \kappa)) $ so the sup over all $\zeta < i$, $\varrho \in N_{\delta,\zeta+1}$, is an element of $[\sup(N_{\delta,i} \cap \kappa), \sup(N_{\delta,i+1} \cap \kappa))$.
$p_{i, \varrho} = \bigcup\{ p_{\sigma_\zeta, \varrho \rest \zeta} : \zeta < i\}$.
Moreover 
$p_{i, \varrho}  \Vdash x_{i, \varrho} = \name{x} \rest \gamma_{i, \varrho}$.
We can compute $\bar{p}_i$, $\bar{x}_i$, $\bar{\gamma'}_i$ from $(\seq{N_{\delta, j}}{j < i}, {\bf c})$. 
 At the same time the set ${}^{\sigma_i} 2$ and ranges  of the sequences 
$\bar{p}_i$, $\bar{x}_i$ and $\bar{\gamma'}_i$ are subsets of $N_{\delta, i+1}$ by ($\ast)_{\gamma_{\delta, i}}$(f).
So carried the induction.

For $\varrho \in {}^\sigma 2$, we let 
\[p_{\sigma, \varrho} = \bigcap\{p_{\sigma_i, \varrho \rest \sigma_i} : i < \cf(\sigma)\}.
\]

By $(\odot)_2(e)$, for each $\varrho \in {}^{\sigma} 2$,
\[p_{\sigma, \varrho} \Vdash  \bigcup\{\gamma'_{\sigma_i, \varrho}: i < \cf(\sigma)\} \in \name{D}.
\]

By $(\odot)_2$(f), for any $\varrho \in {}^\sigma 2 $ such that
for any $i< \cf(\sigma)$, $\varrho \rest \sigma_i \in N_{\delta, i+1}$.
\[
\gamma'_{\sigma, \varrho} = \bigcup\{\gamma'_{\eps, \varrho \rest \eps} : \eps < \sigma\}=\sup(N \cap \kappa) = \delta.
\]
By $(\odot)_{3,i}$ for $i < \cf(\sigma)$ , any  $\varrho \in {}^\sigma 2$ fulfils for any $i< \cf(\sigma)$, $\varrho \rest \sigma_i \in N_{\delta, i+1}$.

We define $(f_1, f_2) \in {\mathcal F}_{\sigma, \delta}$ by letting for $i < \cf(\sigma)$, $\varrho \in {}^{\sigma_i} 2$,
\begin{equation*}\label{f1}
f_1(\varrho) =  \tr(p_{\dom(\varrho), \varrho}).
\end{equation*}
\begin{equation*}\label{f2}
f_2(\varrho) = x_{\dom(\varrho), \varrho}.
\end{equation*}
By the definition of the continuation of the tree embeddings we have:
For any $\varrho \in {}^{\sigma} 2$:
$\bar{f}_1(\varrho) =  \tr(p_{\dom(\varrho), \varrho})$,
$\bar{f}_2(\varrho) = x_{\dom(\varrho), \varrho}$, by $\odot_2$(i) we have
$\dom(\varrho)) = \sigma$ and  $\dom(\tr(p_{\dom(\varrho), \varrho})) = \delta$ if for any $i < \cf(\sigma)$, $\varrho \rest \sigma_i \in N_{\delta, i+1}$.

Now by \autoref{Bernstein_lemma}  there is some $\varrho \in {}^\sigma 2$ 
 with for any $i < \cf(\sigma)$, $\varrho  \rest \sigma_i \in N_{\delta, i+1}$
such that 
\[p_{\sigma, \varrho} \Vdash \bar{f}_1(\varrho) = \name{\eta} \rest \delta \wedge h_\delta(\bar{f_1}(\varrho)) = \name{\nu}_\delta = \bar{f_2}(\varrho) = x_{\sigma, \varrho} = \name{x} \rest \delta.
\]
For the very last equality relation we use $(\odot)_{2}$(d). So $q = p_{\sigma, \varrho}$ and $\delta$ are as in \eqref{odot6.2}.
\end{proof}

Now  Kanamori's premise on iterability is true in the one-step extension:

\begin{corollary} We assume $\aleph_1 \leq \kappa = \kappa^{<\kappa}$.
\begin{enumerate}
\item[(a)] For $\kappa = \aleph_1$ for any stationary $S, W$, we have $\qbW \Vdash \diamondsuit_{\aleph_1}(S)$.
\item[(b)] For $\kappa$ that is not a strong limit,  $\qbW \Vdash \diamondsuit_{\kappa}$.
\end{enumerate}
\end{corollary}

\begin{proof} (a) Any stationary $S \subseteq \aleph_1$ is in the approachability ideal. (b) By \autoref{successor_exponent_2}, there is a stationary set $S \subseteq \kappa \cap \cof(\cf(\sigma))$ with $S \in \checki[\kappa]$.
\end{proof} 

This concludes the proof of \autoref{cor2}.

In the next proposition we show that the Bernstein technique \autoref{Bernstein_lemma} at $2^{\sigma} > \kappa$ may  provide a name of a collapse of $2^\sigma$ to $\kappa$ under additional hypotheses. 

\begin{proposition}\label{l4(2)}
If  $W$ is stationary in $\kappa$ and there is a cardinal $\sigma < \kappa$ such that 
$2^\sigma > \kappa $, $(2^{\sigma})^{2^{<\sigma}} \leq 2^\sigma$, and   $2^{<\sigma} < \kappa$, 
then $\qbW$ collapses this $2^\sigma$  to $\kappa$.

\end{proposition}
 
\begin{proof} We fix a regular $\chi> \beth_\omega(\kappa)$ and let $\mathcal H(\chi) = (H(\chi), \in, <^*_\chi)$.
 Now we use \autoref{Bernstein_lemma} for $\kappa$ from their being now $\kappa' = 2^{\sigma}$.
 Using \autoref{Bernstein_lemma} we have for $\delta \in \kappa \cap \cof(\cf(\sigma))$, a function
  $h_\delta \colon {}^\sigma 2 \to {}^\delta 2$ such that if $f \colon {}^{\sigma >}2 \to {}^{\delta>} 2$ is one-to-one $\sigma$-tree embedding of height $\delta$, then $h_\delta''(\range(\bar{f})) = 2^{|\sigma|}$.

Now define a sequence of $\seq{\name{\tau}_\delta}{\delta \in \kappa \cap \cof(\cf(\sigma))}$ of 
of $\bQ$-names letting
\[
\qbW \Vdash (\forall \delta \kappa \cap \cof(\cf(\sigma)) (\name{\tau}_\delta  = h_\delta(\name{\eta} \rest \delta)
\]

Given $p \in \qbW$ and $\alpha\in 2^{\sigma}$, we
have to produce some $q \geq p$ and some $\delta \in W$ such that
\[q \Vdash \name{\tau}_\delta = \alpha.
\]
To this end, we build a tree of conditions
\[\seq{p_{\sigma_i, \varrho} }{i < \cf(\sigma), \varrho \in {}^{\sigma_i} 2}.
\]
such that, letting 
$p_{\sigma, \varrho} = \bigcap\{p_{\sigma_i, \varrho \rest \sigma_i} : i < \cf(\sigma\}$ we have for
any $\alpha' \in 2^{\sigma}$ there is some  $\varrho \in {}^\sigma 2$ with
\[p_{\sigma, \varrho} \Vdash \mbox{``}
h_\delta(\name{\eta} \rest \delta) = \alpha' \mbox{''},
\] 
so in particular for $\alpha'= \alpha$.

 %%%% out
We construct a tree of conditions
$p_{\sigma, \varrho}$, $\varrho \in {}^\sigma 2$ and a $\sigma$ embedding $f_1$ of height $\delta$ sending $\varrho  $ to $\tr(p_{\dom(\varrho), \varrho})$.

By induction on $i \leq \cf(\sigma)$
we chose a triple $(\bar{p}_i, \bar{\gamma}_i, \delta_i)$ where $\bar{p}_i = \seq{p_{\sigma_i, \varrho}}{\varrho \in {}^{\sigma_i} 2}$,

such that for any $i<\cf(\sigma)$ for any $j < i$ the following holds:
\begin{enumerate}\item[$(\odot)$] 
\qquad $p_i$ is the $<^*_\chi$-least element with (a) to (f), where
\begin{enumerate} 

\item[(a)] $\bar{p}_i = \seq{p_{\sigma_i, \varrho}}{\varrho \in {}^{\sigma_i} 2}$;

\item [(b)] For each $i< \cf(\sigma)$ and each $\varrho \in {}^{\sigma_i} 2$, for any $j < i$, $p_{\sigma_j, \varrho \rest \sigma_j} \leq p_{\sigma_i, \varrho}$;

\item[(c)] for any $\varrho \in {}^{\sigma_i} 2$, $p_{\sigma_i, \varrho} \geq p$;

\item[(d)] for any $\varrho \in {}^{\sigma_i} 2$, $\lg(\tr(p_{\sigma_i, \varrho})) \geq \gamma_{\sigma_i,\varrho}$;

\item[(e)] $\seq{\tr(p_{\sigma_i, \varrho})}{\varrho \in {}^{\sigma_i} 2}$ consists of pairwise $\trl$-incomparable  elements of $p \cap 2^{<\delta}$;
 
\item[(f)] We let $\delta_i = \sup \{
\lg(\tr(p_{\sigma_i, \varrho})) : \varrho \in {}^{\sigma_i} 2\}$. 
For any $\varrho \in {}^{\sigma_{i+1}} 2$, $\lg(\tr(p_{\sigma_{i+1}, \varrho})), \gamma_{\sigma_{i+1}, \varrho} \in[\delta_i+1,  \kappa)$. 
\end{enumerate}
\end{enumerate}

There are no problems in the inductive choice. Again $2^{\sigma_i} < \kappa$ is crucial for (f). Note that
the $F_i:= \{\tr(p_{\sigma_i,\varrho} : \varrho \in {}^{\sigma_i} 2\}$ is a front of $p_i := \bigcup\{p_{\sigma_i, \varrho} : \varrho \in {}^{\sigma_i}2\}$. 

Now suppose that the induction is performed. We let for $i < \cf(\sigma)$, $\varrho \in {}^{\sigma_i} 2$, 
$\tr(p_{\sigma_i,\varrho}) := f_1(\varrho)$. Thus $f_1$ defines a $\sigma$-embedding of height $\delta$.
We let $f_2(\varrho) = \alpha' $ for any $\varrho \in {}^{\sigma_i} 2$ for any $i < \cf(\sigma)$.

 By the choice of $h_\delta$, 
there is some $\varrho \in {}^\sigma 2$ such 
\[p_{\sigma, \varrho} \Vdash h_\delta(\bar{f_1}(\varrho)) =  h_\delta(\name{\eta} \rest \delta) = \bar{f_2}(\varrho) = \alpha'.
\]
\end{proof}

\begin{remark}
\autoref{l4(2)}  is proved differently for ordinary $\kappa$-Sacks in \cite{MdSh:1191}, where Solovay partitions of stationary sets in pairwise disjoint stationary sets are used and Clause (2) \autoref{qb} is used. 
\end{remark}

\def\cprime{$'$} \def\germ{\frak} \def\scr{\cal} \ifx\documentclass\undefinedcs
  \def\bf{\fam\bffam\tenbf}\def\rm{\fam0\tenrm}\fi % f**k-amstex!
  \def\defaultdefine#1#2{\expandafter\ifx\csname#1\endcsname\relax
  \expandafter\def\csname#1\endcsname{#2}\fi} \defaultdefine{Bbb}{\bf}
  \defaultdefine{frak}{\bf} \defaultdefine{=}{\B} % doublef**k-amstex!!
  \defaultdefine{mathfrak}{\frak} \defaultdefine{mathbb}{\bf}
  \defaultdefine{mathcal}{\cal} \defaultdefine{implies}{\Rightarrow}
  \defaultdefine{beth}{BETH}\defaultdefine{cal}{\bf} \def\bbfI{{\Bbb I}}
  \def\mbox{\hbox} \def\text{\hbox} \def\om{\omega} \def\Cal#1{{\bf #1}}
  \def\pcf{pcf} \defaultdefine{cf}{cf} \defaultdefine{reals}{{\Bbb R}}
  \defaultdefine{real}{{\Bbb R}} \def\restriction{{|}} \def\club{CLUB}
  \def\w{\omega} \def\exist{\exists} \def\se{{\germ se}} \def\bb{{\bf b}}
  \def\equivalence{\equiv} \let\lt< \let\gt>

%\bibliographystyle{plain}
%\bibliography{../sh/lit,../sh/listb, ../11forschung/bibliography}
\end{document}